\documentclass[reqno]{amsart}
\usepackage{amsthm, amsfonts, amssymb, color}
\usepackage{mathrsfs}
\usepackage{geometry}
\usepackage[colorlinks=true, linkcolor=blue]{hyperref}

\newcommand{\R}{\mathbb R}
\newcommand{\C}{\mathbb C}

\newcommand{\N}{\mathbb{N}}
\newcommand{\les}{\lesssim}

\def\d{{\rm d}}

\newtheorem{theorem}{Theorem}[section]
\newtheorem{proposition}[theorem]{Proposition}

\newtheorem{lemma}[theorem]{Lemma}

\newtheorem{remark}[theorem]{Remark}

\numberwithin{equation}{section}

\title[The $L^p$ boundedness of wave operators for finite rank perturbations]{The $L^p$ boundedness of wave operators for the Laplace operator with finite rank perturbations}
\author{Han Cheng, Shanlin Huang, Avy Soffer, Zhao Wu}

\address {Han Cheng, Institute of Applied Physics and Computational Mathematics, Beijing, 100088, People's Republic of China}
\email{chmathh@163.com}

\address {Shanlin Huang, School of Mathematics (Zhuhai), Sun Yat-sen University, Zhuhai 519082, Guangdong, China}
\email{huangshlin6@mail.sysu.edu.cn}

\address {Avy Soffer, Department of Mathematics, Rutgers University, 110 Frelinghuysen Rd., Piscataway, NJ, 08854, USA}
\email{soffer@math.rutgers.edu}

\address {Zhao Wu, Faculty of Mathematics and Statistics, Hubei University, Wuhan, 430062, People's Republic of China}
\email{wuzhao@hubu.edu.cn}


\begin{document}
	\maketitle
	\begin{abstract}
		This paper investigates the $L^p$ boundedness of wave operators for the Laplace operator  with  finite rank perturbations 
		$$
		H=-\Delta+\sum\limits_{i=1}^N\langle\cdot\,, \varphi_i\rangle \varphi_i \qquad \mbox{on}\,\,\, \R^d.
		$$
		For dimensions $d\ge 3$, we prove that  the  wave operators $W_\pm(H,H_0)$ are bounded on $L^p$   for the full range $1\le p\le \infty$. This extends the work of Nier and the third author \cite{NS}  by resolving the previously  unexplored question of boundedness at the endpoint cases  $p=1$ and $p=\infty$. 
		In lower dimensions $d = 1, 2$, we establish the $L^p$-boundedness of the wave operators for the first time. Furthermore, we reveal an intriguing dichotomy in the endpoint case  $p = 1$:  
		\begin{itemize}
			\item If  $\int_{\mathbb{R}^d} \varphi_i(x) \, \d x = 0$  holds for every $1\le i\le N$, then the wave operators are bounded on $L^p(\mathbb{R}^d)$ for all $1 \leq p \leq \infty$.  
			\item If there exists at least one $i$ ($1\le i\le N$)   such that $\int_{\mathbb{R}^d}\varphi_i(x)\d x\ne0$, then the wave operators remain bounded for $1 < p < \infty$ and satisfy weak type $(1,1)$ estimates, but fail to be bounded on $L^1(\mathbb{R}^d)$.  
		\end{itemize}
		
		
	\end{abstract}
	

	\section{Introduction and main results}
	\setcounter{equation}{0}
	\subsection{Background}
	
	In the context of time-dependent scattering theory, we consider the Hilbert space $\mathcal{H}:=L^2(\mathbb{R}^d)$, where  the free Schr\"{o}dinger operator is given by $H_0:=- \Delta$ and the perturbed Schr\"{o}dinger operator  is defined as $H := - \Delta + V$. Here $V$ is a measurable, bounded real function on $\mathbb{R}^d$ that decays sufficiently rapidly at infinity. It is well known that the wave operators, defined by the strong limits in $L^2(\mathbb{R}^d)$
	$$
	W_\pm(H,H_0):= s-\lim_{t \to \pm \infty} e^{itH} e^{-itH_0},
	$$
	exist and are asymptotically complete, see e.g. in \cite{Reed-Simon-III}.
	Consequently,  the scattering operator $S := W^*_+ W_-$ is  unitary on $\mathcal{H}$.
	
	Among the numerous properties of wave operators, their mapping behavior between $L^p$ spaces for $p\ne 2$ has been a subject of significant interest since Yajima's seminal work \cite{Yajima95,Yajima95-even}. Over the past three decades, this topic has been extensively studied, with results varying depending on the spatial dimension and the presence or absence of eigenvalues and resonances at zero energy. For instance, sharp results for  $d=3$ can be found in \cite{Beceanu},
	while cases for $d=1,2$ are treated in \cite{AF06,Gal-Yaj-00,JY02,Weder99,Yajima99}.  When zero energy is not regular, we refer to \cite{EGG18,FY06,GG16,GG17,Yajima06,Yajima16,Yajima22-2d,Yajima22-4d}  and the references  therein.
	More recently, the $L^p(\mathbb{R}^n)$ boundedness of the wave operators  Schr\"{o}dinger operator with inverse square potential was established in \cite{MSZ}, and \cite{Ya21} with point
	interactions.  For the higher order  Schr\"{o}dinger operator $H=(-\Delta)^m+V$,  we refer to  \cite{GY23,GG21,MWY22,MWY23-1,MWY23-2}.

	\textbf{Aim and Motivation.} 
	The main goal of this paper is to investigate the $L^p$ boundedness of the wave operator for the  Schr\"{o}dinger operator  under finite rank perturbations:
	\begin{equation}\label{eq-H}
		H=H_0+\sum_{j=1}^NP_j, \,\,\,\,\, H_0=-\Delta, \,\,\, P_j=\langle\cdot\,, \varphi_j\rangle \varphi_j.
	\end{equation}
	The theory of finite rank perturbations can be traced back to a seminal paper in
	1910 by Weyl \cite{Wey}, where they were introduced as a tool to determine the spectrum of
	Sturm-Liouville operators. They also  arise in a number of problems in mathematical physics.
	For example, Simon and Wolff \cite{SW} found a relationship between rank one perturbation and discrete random Sch\"{o}dinger operators, and applied it to the Anderson localization conjecture.
	We note that \eqref{eq-H} can be viewed as the Sch\"{o}dinger Hamiltonian with a large number of obstacles, and each obstacle is described by a rank one perturbation.
	It deserves to mention that despite the formal simplicity, the study of this restrictive perturbation is very rich and interesting, we refer to \cite{Simon-1995} for the spectral analysis and the survey paper \cite{Liaw} for other problems.

	Our work is partly inspired by the result of Nier and the third author \cite[Theorem 2.1]{NS}, which established that  under rank one perturbations, the wave operators  $W_\pm(H_\alpha,H_0)$  are bounded on $L^p(\R^d)$ for   $1<p<\infty$
	in dimensions  $d\ge 3$.  This result holds under a spectral  assumption (see \eqref{eq-spectral condition}) and the following decay and smoothness condition
	\begin{equation*}\label{eq1.6}
		\langle x\rangle^s \langle D\rangle^{\beta}\varphi\in L^2(\mathbb{R}^d), \,\,\text{for}\,\,s>[\frac{d}{2}]+2,\,\,\,\beta\ge\max\{\frac{d-3}{2}, \frac12[\frac{d}{2}]\}.
	\end{equation*}
	However,  the end-point cases $p=1$ and $p=+\infty$ remain unknown, and the problem has not yet been resolved for dimensions $d=1,2$.
	In light of existing results for  Schr\"{o}dinger operators mentioned above, a natural question arises: Can the boundedness of the wave operator for the Hamiltonian \eqref{eq-H}  extended to all dimensions?
	This stu\d y aims to bridge this gap by establishing $L^p$ boundedness in  arbitrary dimensions, with a particular focus on proving a dichotomy for the cases $d=1$ and $d=2$. 
	
	\subsection{Main results}
	We begin by examining rank one perturbations. Specifically, consider
	\begin{equation}\label{eq-rank one perturbation}
		H_\alpha=H_0+\alpha\langle\cdot,\varphi\rangle\varphi,\ \ \ H_0=-\Delta,\ \ \ \alpha>0,
	\end{equation}
	where $\varphi\in L^2.$ Throughout the paper, $\langle\cdot,\cdot\rangle$ denotes the standard  inner product in $L^2(\R^d).$
	
	Before stating our assumptions on $\varphi,$ we note that the following function 
	\begin{equation}\label{eq-F function}
		F(z):=\langle (-\Delta-z)^{-1}\varphi,\varphi\rangle=\int\frac{d\mu_{H_0}^\varphi(\lambda)}{\lambda-z},\quad z\in\C\setminus[0,\infty),
	\end{equation}
	plays a key role in our analysis. Moreover, if $\varphi$  belongs to the weighted $L^2$ space $L^2_\sigma$ with $\sigma>\frac{1}{2},$ the well-known limiting absorption principle (see e.g. \cite{Agmon}) ensures that the boundary values
	\begin{equation}\label{eq-F function-limiting absorption}
		F^\pm(\lambda):=\langle (-\Delta-\lambda\mp i0)^{-1}\varphi,\varphi\rangle
	\end{equation}
	are well defined when $\lambda\in\mathbb{R}\setminus\{0\}$ and coincide on $(-\infty,0).$
	
	\textbf{Condition} $H_\alpha.$ Let $\alpha>0$ and let $\varphi$ be a $L^2$ normlized real valued function on $\mathbb{R}^d$ satisfying:
	
	(i) \textbf{Decay and smoothness}: For $d\ge 1,$ for any $\delta>d+2$, $\varphi(x)$ satisfies
	\begin{equation}\label{eq-decay condition}
		|\varphi(x)|\lesssim \langle x\rangle^{-\delta}.
	\end{equation}
	Additionally, when $d\ge 2,$ we assume that $\varphi(x)\in C^{\beta_0}(\mathbb{R}^d)$ with $\beta_0=\big[\frac{d}{2}\big]$, and satisfies 
	\begin{equation}\label{eq-smoothness condition}
		|\partial^\beta_x\varphi(x)|\lesssim \langle x\rangle^{-\delta},\ \ \ |\beta|\le\beta_0.
	\end{equation}
	
	(ii) \textbf{Spectral assumption}: For $d\ge 1,$ there exists an absolute constant $c_0>0$ such that
	\begin{equation}\label{eq-spectral condition}
		|1+\alpha F^\pm(\lambda^2)|\ge c_0,\,\,\, \lambda>0,
	\end{equation}
	where the functions $F^\pm(\lambda^2)$ are given by \eqref{eq-F function-limiting absorption}. 
	
	These conditions ensure that the spectrum of  $H_\alpha$ is absolutely continuous, as demonstrated in \cite[Lemma 2.6]{NS} and \cite[Lemma B.1]{CHZ23}. Since the result differs between $d\ge 3$ and $d=1,2$, we present them separately.
	
	\begin{theorem}\label{thm-main result-rank one-d>3}
		For dimensions $d\ge 3$, if $\varphi$ satisfies \textbf{Condition $H_\alpha$}, then the wave operators $W_\pm(H_\alpha,H_0)$ are bounded on $L^p(\mathbb{R}^d)$ for all $1\le p\le \infty.$
	\end{theorem}

	\begin{theorem}\label{thm-main result-rank one-d=1,2}
		For dimensions $d=1,2$, under \textbf{Condition $H_\alpha$}, the following holds:
		
		\begin{enumerate}
			\item[(i)] 
			When $\int_{\mathbb{R}^d}\varphi(x)\d x=0,$ the wave operators $W_\pm(H_\alpha,H_0)$ are bounded on $L^p(\mathbb{R}^d)$ for the full range $1\le p\le \infty.$
			
			\item[(ii)]  When $\int_{\mathbb{R}^d}\varphi(x)\d x\neq0,$ the wave operators $W_\pm(H_\alpha,H_0)$ exhibit:
			\begin{itemize}
				\item Boundedness on $L^p(\mathbb{R}^d)$ for all $1< p<\infty$;
				
				\item Weak type (1, 1) boundedness from $L^1(\mathbb{R}^d)$ to $L^{1,\infty}(\mathbb{R}^d)$; 
				
				\item Failure of boundedness on $L^1(\mathbb{R}^d)$ and $L^{\infty}(\mathbb{R}^d)$ in dimension $d=1$ and on $L^1(\mathbb{R}^d)$  when $d=2$.
				
			\end{itemize}
		\end{enumerate}
	\end{theorem}
	
	We make the following remarks on Theorem \ref{thm-main result-rank one-d>3} and \ref{thm-main result-rank one-d=1,2}.
	\begin{itemize}
		\item[($\textbf{a}_1$)]  Theorem \ref{thm-main result-rank one-d>3} improves the result of Nier and the third author \cite[Theorem 2.1]{NS} in the following sense: under similar decay and smoothness assumptions on  $\varphi$, the  boundedness of $W_\pm(H_\alpha,H_0)$ remains valid at the endpoint cases $p=1$ and $p=\infty$ when $d\ge 3$. 
		To the authors' best knowledge, Theorem  \ref{thm-main result-rank one-d=1,2} appears to be the first result establishing the $L^p$ boundedness of the wave operator $W_\pm(H_\alpha,H_0)$ in dimensions $d=1,2$.
		
		\item[($\textbf{a}_2$)] For spatial dimensions $d=1, 2$, we compare our results with those for the Schr\"{o}dinger operator $H=-\Delta+V$, Weder \cite[Theorem 1.1]{Weder99} first proved boundedness in $L^p(\mathbb{R})$ for $1<p<\infty$ and observed that the case $p=1$ holds if  the Jost solutions of $-f''+Vf=\lambda^2f$ satisfy  $\lim_{x\rightarrow-\infty}f_1(x,0)=1$ and $V$ is exceptional (the  Wronskian $W(0)=0$), see also in \cite{AF06,Gal-Yaj-00}. The  2D case is more complicated:  Erdog\u{a}n, Goldberg and Green \cite{EGG18} showed that the wave operators are bounded on $L^p(\mathbb{R}^2)$ for $1<p<\infty$ if there is an $s-$wave resonance or an eigenvalue only at zero; Yajima \cite{Yajima22-2d} proved that if there is a $p-$wave resonance at zero, then the wave operator is bounded on $L^p(\mathbb{R}^2)$ for $1<p\le 2$ and unbounded for $2<p<\infty$, while the problem for $p=1$ is left open.  
		In contrast, our work presents a dichotomy for boundedness of $W_\pm(H_\alpha,H_0)$  at  $p=1$,  depending on whether  $\int_{\mathbb{R}^d}\varphi(x)\d x$ ($d=1,2$) vanishes. Notably, when $\int_{\mathbb{R}^d}\varphi(x)\d x\ne 0$, $d=1,2$, the wave operator  $W_\pm(H_\alpha,H_0)$ is unbounded on  $L^1(\mathbb{R}^d)$. It is noteworthy that by \cite[Theorem 1.1]{CHZ23}, the dispersive estimate
		\begin{equation}\label{eq-disp}
			\|e^{-itH_\alpha}P_{ac}(H_\alpha)\|_{L^1-L^{\infty}}\leq C t^{-\frac{d}{2}},\quad\,\,\,\,t>0   
		\end{equation}
		remains valid regardless of whether $\int_{\mathbb{R}^d}\varphi(x)\d x$  vanishes or not.  This also demonstrates that the boundedness of the wave operators   $W_\pm(H_\alpha,H_0)$  for $p=1$ is strictly stronger than the dispersive estimate \eqref{eq-disp}.
			
			\item[($\textbf{a}_3$)] Our strategy to prove  Theorem \ref{thm-main result-rank one-d>3}-\ref{thm-main result-rank one-d=1,2} is  quite different from the method employed in \cite{NS}. The proof in \cite{NS} depends critically on Yajima’s earlier results \cite{Yajima95,Yajima95-even} and an application of the  H\"{o}rmander multiplier theorem to the function $\frac{1}{1+\alpha F^{\pm}(|\xi|^2)}$. In contrast, we exploit the resolvent and analyze the pointwise behavior of the kernel of $W_\pm(H_\alpha,H_0)$ through a detailed analysis of oscillatory integrals.

		\end{itemize}

		Now  we consider the finite rank perturbations
		\begin{equation}\label{eq-finite rank perturbation}
			H=H_0+\sum_{i=1}^NP_j,\ \ \ H_0=-\Delta,\ P_j=\langle \cdot,\varphi_j\rangle\varphi_j.
		\end{equation}
		In this case, we assume that  $\varphi_1,\cdots,\varphi_N$ are $L^2$-normalized and mutually orthogonal, i.e.
		\begin{equation}\label{eq-normalized assumption}
			\langle \varphi_i,\varphi_j\rangle=\delta_{ij},\ \ \ 1\le i,j\le N,
		\end{equation}
		where $\delta_{ij}$ denotes the Kronecker delta. In addition to   decay and smoothness conditions on each $\varphi_j,$ we require an analog of the spectral assumption \eqref{eq-spectral condition}. Instead of the scalar function defined in \eqref{eq-F function-limiting absorption}, we introduce the $N\times N$ matrix
		\begin{equation}\label{eq-F matrix}
			F_{N\times N}(z)=\big(f_{ij}(z)\big)_{N\times N}, \ \ \ f_{ij}(z)=\langle(-\Delta-z)^{-1}\varphi_i,\varphi_j\rangle,\ z\in\mathbf{C}\setminus[0,\infty).
		\end{equation}
		For  $\varphi_i\in L^2_\sigma$ with $\sigma>\frac{1}{2}$ ($1\le i\le N$),  the boundary values
		\begin{equation}\label{eq-F matrix-limiting absorption}
			F^\pm_{N\times N}(\lambda^2):=F_{N\times N}(\lambda^2\pm i0),\, \, \,\lambda>0
		\end{equation}
		are well-defined.
		Motivated by \eqref{eq-spectral condition}, we impose the following spectral assumption: There is some absolute constant $c_0>0$ such that
		\begin{equation}\label{eq-spectral condition-finite rank}
			|\det(\mathbb{I}_{N}+F^{\pm}_{N\times N}(\lambda^2))|\ge c_0>0
		\end{equation}
		holds for any $\lambda>0$. Indeed, this assumption implies $\sigma(H)=\sigma_{ac}(H)=[0,+\infty),$ see Lemma B.1 in \cite{CHZ23}.
		
		\begin{theorem}\label{thm-main result-finite rank}
			Assume that each $\varphi_i\ (1\le i\le N)$ satisfies (i) of \textbf{Condition $H_\alpha$} and that the spectral assumption \eqref{eq-spectral condition-finite rank} holds.

			(i) For $d\ge3,$ the wave operators $W_\pm(H,H_0)$ are bounded on $L^p(\R^d)$ for all  $1\le p\le\infty$.
			
			(ii) For  $d=1,2$, if $\int_{\mathbb{R}^d}\varphi_i(x)\d x=0$ holds for every $1\le i\le N$, then the wave operators $W_\pm(H,H_0)$ remain bounded on $L^p(\R^d)$ for all  $1\le p\le\infty$.
			
			(iii) For  $d=1,2$, if there exists at least one $i$ ($1\le i\le N$)   such that $\int_{\mathbb{R}^d}\varphi_i(x)\d x\ne0$, then the wave operators $W_\pm(H,H_0)$ are bounded on $L^p(\R^d)$  for $1< p<\infty$ and from $L^1$ to $L^{1,\infty}$, but fails to be bounded  at the endpoints $p=1, \infty$ when $d=1$ and $p=1$ when $d=2$.
		\end{theorem}
		
		Some remarks on Theorem \ref{thm-main result-finite rank}  are as follows:
		\begin{itemize}
			\item[($\textbf{b}_1$)] 
			As far as we are aware, this is the first result concerning $L^p$ boundedness of the wave operators for \eqref{eq-finite rank perturbation} in any dimensions. Previous related work by  Nier and the third author \cite[Theorem 1.1]{NS} addressed   $L^p-L^{p'}$ estimates for the propagator $e^{-itH}$ in the range $1<p<\infty$, 
			where $\varphi_i(x) = \varphi(x - \tau_i)$ and the spreading condition $|\tau_i - \tau_j| \geq \tau_0$ is assumed. 
			The endpoint case $p=1$ was   recently resolved in \cite{CHZ23} for more general $\varphi_i$.
			
			\item[($\textbf{b}_2$)] We shall apply the Aronszjn-Krein type formula for finite rank perturbations (see \cite[Proposition 4.1]{CHZ23}). This allows us to reduce the problem to the rank one case by studying the inverse of the matrix
			$\mathbb{I}_{N}+F_{N\times N}(\lambda^2)$, or equivalently, the inverse operator $(I + P R_0^\pm(\lambda^2) P)^{-1}$. Notably, the analysis becomes more intricate for dimensions $d = 1$ and $d = 2$ due to the singularities of $R_0^\pm(\lambda^2)$ at $\lambda=0$. Our approach combines two essential components: (i) a  suitable direct sum decomposition of the space $PL^2$ to isolate the singular terms, and (ii) a matrix-based approach to obtain the inverse of $I + P R_0^\pm(\lambda^2) P$. The idea is inspired by the classical work of Jensen and Kato \cite{{J80}}.
			
			\item[($\textbf{b}_3$)]
			As a natural extension, it would be interesting to extend our results to the case $N=\infty$, i.e., 
			\begin{equation*}
				H=H_0+\sum_{i=1}^{\infty}\lambda_jP_j,
			\end{equation*}
			where the coefficients satisfy  $\lambda_j\ge 0$ and $\{\lambda_j\}\in l^p$ for some $p\ge 1$ (trace  class).  Another interesting  direction is the case of random coefficients, which model disordered systems. For example,  the case where  $\lambda_j\in\{0, 1\}$ are  independent Bernoulli random variables.
		\end{itemize}

			
		
		The rest of the paper is organized as follows:
		Section \ref{sec2} introduces the stationary representation of wave operators,  provides expansions of the free resolvent, as well as some integral estimates that will be frequently used in subsequent sections.
		Section \ref{sec3} is devoted to the proof of  Theorem \ref{thm-main result-rank one-d>3}.
		Section \ref{sec4} presents the proof of  Theorem \ref{thm-main result-rank one-d=1,2}. In section \ref{sec5}, we give the proof of Theorem \ref{thm-main result-finite rank}.

		\section{Preliminaries}\label{sec2}
		\setcounter{equation}{0}
		\subsection{Stationary representation.}
		This subsection presents the explicit representation of wave operators for both rank one and finite rank perturbations.
		
		For rank one perturbations \eqref{eq-rank one perturbation}, let $R_\alpha^\pm(\lambda^2)=(H_{\alpha}-\lambda^2\mp i0)^{-1}$. By the  Aronszajn-Krein formula  (see e.g.  \cite{Simon-1995}), we have
		\begin{equation}\label{eq-AK formula-rank one}
			R_\alpha^\pm(\lambda^2)=R_0^\pm(\lambda^2)-\frac{\alpha}{1+\alpha F^\pm(\lambda^2)}R_0^\pm(\lambda^2)\varphi\langle R_0^\pm\cdot,\varphi\rangle,
		\end{equation}
		where $F^\pm(\lambda^2)$ is given by \eqref{eq-F function-limiting absorption}.
		Recall that for Schr\"{o}dinger operators $H=H_0+V$,  the wave operator $W_\pm=W_\pm(H_0+V,H_0)$ has the form 
		\begin{equation}\label{eq-stationary representation-V}
			W_\pm=I-\frac{1}{\pi i}\int_0^\infty\lambda(H_0+V-\lambda^2\pm i0))^{-1}V\big((H_0-\lambda^2-i0))^{-1}-(H_0-\lambda^2+i0))^{-1}\big)\d\lambda.
		\end{equation}
		Substituting $V=\alpha\langle\cdot,\varphi\rangle\varphi$ and using \eqref{eq-AK formula-rank one}, we obtain the stationary representation of the wave operators for rank one perturbations (see  also in \cite{NS}):
		\begin{equation}\label{eq-stationary representation-one rank}
			\begin{split}
				W_\pm(H_\alpha,H_0)&=I-\frac{1}{\pi i}\int_0^\infty\frac{\alpha\lambda}{1+\alpha F^{\mp}(\lambda^2)}R_0^\mp(\lambda^2)\varphi\langle (R_0^+(\lambda^2)-R_0^-(\lambda^2))\cdot,\varphi\rangle \d\lambda \\
				&:=I-\mathcal{W}_{\pm}. 
			\end{split}
		\end{equation}

		For finite rank perturbations, let $P=\sum\limits_{j=1}^NP_j$.  From  the following resolvent identity 
		\begin{equation}\label{eq-resolvent identity-finite rank}
			R^\pm(\lambda^2)P=R_0^\pm(\lambda^2)P-R^\pm(\lambda^2)PR_0^\pm(\lambda^2)P,
		\end{equation}
		we derive the Aronszajn-Krein formula for finite rank perturbations:
		\begin{equation}\label{eq-AK formula-finite rank}
			R^\pm(\lambda^2)P=R_0^\pm(\lambda^2)P\left(I+ PR_0^\pm(\lambda^2)P\right)^{-1}P =\sum_{i,j=1}^Ng_{i,j}^\pm(\lambda^2)R_0^\pm(\lambda^2)\varphi_i \, \langle \cdot,\varphi_j\rangle,
		\end{equation}
		where $G^\pm(\lambda^2)=(g^\pm_{i,j}(\lambda^2))_{N\times N}$ is the inverse matrix of 
		\begin{equation}\label{eq: def of matrix A}
			A^\pm(\lambda)=\mathbb{I}_{N}+F^\pm_{N\times N}(\lambda^2)=(a_{i,j}^\pm(\lambda))_{N\times N}
		\end{equation}
		with 
		\begin{equation*}\label{eq-a_ij}
			a_{ij}^\pm(\lambda)=\delta_{ij}+\langle R^\pm_0(\lambda^2)\varphi_j,\varphi_i\rangle=\delta_{ij}+f^\pm_{j,i}(\lambda)\ \ \ 1\le i,\, j\le N.
		\end{equation*}
		Plugging \eqref{eq-AK formula-finite rank} into \eqref{eq-stationary representation-V} with $V=P$, we obtain the following stationary representation of wave operators for \eqref{eq-finite rank perturbation}:
		\begin{equation}\label{eq-stationary representation-finite rank}
			\begin{split}
				W_\pm(H,H_0) 
				&=I-\sum_{i,j=1}^N\frac{1}{\pi i}\int_0^\infty \lambda g^\mp_{i,j}(\lambda^2) R_0^\mp(\lambda^2)\varphi_{j}\langle(R_0^+(\lambda^2)-R_0^-(\lambda^2))\cdot,\varphi_i\rangle \d\lambda \\
				&:=I-\sum_{i,j=1}^N\mathcal{W}_{\pm,i,j}.
			\end{split}
		\end{equation}
		
		Since $W_-(H_\alpha,H_0)f = \overline{W_+(H_\alpha,H_0)\bar{f}}$ and the identity operator is obviously bounded in $L^p$ for all $1 \leq p \leq \infty$, it suffices to analyze $\mathcal{W}_{-}$ in the rank-one case and $\mathcal{W}_{-,i,j}$ in the finite-rank case. Furthermore, since the stationary representations of $\mathcal{W}_{-,i,j}$ and $\mathcal{W}_{-}$ share the same form, we can reduce the boundedness analysis in the finite-rank case to the rank-one case.
		
		\subsection{The free resolvent $R_0^\pm(\lambda^2)$}
		
		To analyze the  $L^p$-boundedness of the wave operators  $W_-(H_\alpha,H_0)$ and $W_-(H,H_0)$  using their stationary representations \eqref{eq-stationary representation-one rank} and \eqref{eq-stationary representation-finite rank}, we first review some  basic properties of the free resolvent $R_0(z)=(-\Delta-z)^{-1}$.
		
		Recall that $R^{\pm}_0(\lambda^2)$ are integral operators with kernels as
		\begin{equation}\label{eq2.1}
			R^{\pm}_0(\lambda^2; x,y)=\frac{\pm i}{2\lambda}e^{\pm i\lambda|x-y|}, \,\,\,\,\,\, d=1,
		\end{equation}
		and
		\begin{equation}\label{eq2.2}
			R^{\pm}_0(\lambda^2; x,y)=\frac{\pm i}{4}(\frac{\lambda}{2\pi|x-y|})^{ \frac{d-1}{2}}H^{\pm}_{\frac{d-1}{2}}(\lambda|x-y|), \,\,\,\,\,\, d\ge 2,
		\end{equation}
		where $H^{\pm}_{\frac{d-1}{2}}(z)=J_{\frac{d-1}{2}}(z)\pm i Y_{\frac{d-1}{2}}(z)$ are Hankel functions of order $\frac{d-1}{2}$. Note that when $d$ is odd, $H^{\pm}_{\frac{d-1}{2}}(z)$ admit finite-term closed-form expansions, while when $d$ is even,  the  Hankel functions have integer orders, leading to a logarithmic branch point at zero. Unlike the odd-dimensional case, they cannot be expressed in closed form (see \cite[Chapt.10]{Ol}). In the following, we shall give different representations of \eqref{eq2.2}.

		First, for odd dimensions $d\ge 3$, we have
		\begin{equation}\label{eq2.3}
			R^{\pm}_0(\lambda^2; x,y)=c_d\frac{e^{\pm i\lambda|x-y|}}{|x-y|^{d-2}}\,\sum^{\frac{d-3}{2}}_{k=0}{\frac{(d-3-k)!}{k!(\frac{d-3}{2}+k)!}(\mp2i\lambda|x-y|)^k},
		\end{equation}
		where $c_d=1/(4\pi)^{\frac{d-1}{2}}$.  Using Taylor's expansion for $e^{\pm i\lambda|x-y|}$, it can be written as (see e.g. in  \cite{J80})
		\begin{equation}\label{eq-epansion-small-argument-0}
			R^{\pm}(\lambda^2; x,y)=|x-y|^{2-d}\sum_{j=0}^{\frac{d-3}{2}}\mu_j(\lambda|x-y|)^{2j}+\lambda^{2-d}\sum_{j=0}^{\infty}\mu_j(\pm i)^{j+d-2} (\lambda|x-y|)^j.
		\end{equation}
		While for even dimensions $d\ge 2$, $R^{\pm}(\lambda^2; x,y)$ has the following expansion around zero:
		\begin{equation}\label{eq-epansion-small-argument}
			R^{\pm}(\lambda^2; x,y)=|x-y|^{2-d}\sum_{j=0}^{\infty}B_j(\pm \lambda|x-y|),\qquad 0<\lambda|x-y|\ll1,
		\end{equation}
		where  $B_j(z)$ is given by
		\begin{align*}
			B_j(z)=
			\begin{cases}
				\mu_j z^{2j},&\text{if}\,\, 0\le j\le \frac{d}{2}-2,\\[4pt]
				\upsilon_j z^{2j}+\mu_j z^{2j}\log(z),&\text{if}\,\,j\ge \frac{d}{2}-1.\\[4pt]
			\end{cases}
		\end{align*}
		Here,  $\mu_j$ and $\upsilon_j$ are dimension-dependent constants with  $\mu_j\in \R\setminus\{0\} $ and $\upsilon_j\in \mathbb{C}\setminus{\mathbb{R}}$. 
		
		On the other hand, by the asymptotic behavior of Hankel functions for large arguments,  it follows that
		\begin{equation}\label{eq-resolvent-large-argument-1}
			{R}_0^{\pm}(\lambda^{2};x,y)=\frac{\lambda^{\frac{d-3}{2}}}{|x-y|^{\frac{d-1}{2}}}e^{\pm i\lambda|x-y|}\Phi^\pm(\lambda|x-y|),	\qquad \lambda|x-y|>\frac12,
		\end{equation}
		where $\Phi^\pm(z)\in C^\infty\left((\frac12, \infty)\right)$ satisfies
		\begin{equation}\label{eq-resolvent-large-argument-2}
			\left| \frac{\d^k}{\d z^k}\Phi^\pm(z)\right| \lesssim_k \langle z\rangle^{-k}, \quad  k\in \N_0.
		\end{equation}
		Thus for $d\ge 3$,   using \eqref{eq-epansion-small-argument-0}, \eqref{eq-epansion-small-argument} and \eqref{eq-resolvent-large-argument-1}, we can decompose the  kernel $R_0^\pm(\lambda^2;x,y)$ into
		\begin{equation}\label{eq：exp for  reso-0-1}
			R_0^\pm(\lambda^2;x,y)=c_d|x-y|^{2-d}+r_{d}^\pm(\lambda,|x-y|),    
		\end{equation}
		where the remainder terms  $r_d^\pm(\lambda,|x-y|)$ satisfy,  for $ k=0,1,\cdots, \big[\frac{d+3}{2}\big]$ and $0<\lambda<1$, 
		\begin{equation}\label{eq: est r_d>3}
			\Big|\partial_\lambda^k r_d^\pm(\lambda,|x-y|)\Big|\lesssim 
			\begin{cases}
				\lambda^{1-k}|x-y|^{3-d},\quad& \text{$\lambda|x-y|\le  1$},\\  
				\lambda^{\frac{d-3}{2}-k}(|x-y|^{-\frac{d-1}{2}}+|x-y|^2 ) ,\quad& \lambda|x-y|>1.
			\end{cases}
		\end{equation}
		
		Second, for $d\ge 3$, we can also write
		\begin{equation}\label{eq-free kernel-1}
			R^{\pm}_0(\lambda^2; x,y)=\lambda^{\frac{d-3}{2}}e^{\pm i\lambda|x-y|}\frac{w^{\pm}_{ 0}(\lambda|x-y|)}{|x-y|^{\frac{d-1}{2}}}+e^{\pm i\lambda|x-y|}\frac{w^{\pm}_{ 1}(\lambda|x-y|)}{|x-y|^{d-2}},
		\end{equation}
		where the functions  $w^{\pm}_{0}(z),\, w^{\pm}_{1}(z)$ satisfy, for $s=0,1$, 
		\begin{equation}\label{eq-w estiamte-0702}
			\Big|\frac{\d^k w_s^{\pm}(z)}{\d z^k}\Big|\lesssim |z|^{-k},\quad \,\,\,k\in\mathbb{N}_0.
		\end{equation}
		And for $d=2$, we can express $R_0^\pm(\lambda^2;x,y)$ as 
		\begin{equation}\label{eq-free kernel-3}
			R_0^\pm(\lambda^2;x,y)=e^{\pm i\lambda|x-y|}w^\pm(\lambda|x-y|),
		\end{equation}
		where  $w^\pm(z)\in C^\infty\big((0,+\infty)\big)$ satisfies
		\begin{equation}\label{eq-w estiamte}
			\Big|\frac{\d^k w^\pm(z)}{\d z^k}\Big|\lesssim |z|^{-\frac 12-k},\quad\,\,\,k\in\mathbb{N}_0.
		\end{equation}

		Third, we define a smooth cut-off function $\eta(\lambda)$ such that 
		\begin{equation}\label{eq-cut-off function eta}
			\eta(\lambda)=1,\ \quad 0<\lambda<\frac12, \ \ \ \ \mbox{and} \ \ \  \eta(\lambda)=0,\quad\ \lambda>1.
		\end{equation}
		Let $\widetilde{\eta}=1-\eta$.
		Combining \eqref{eq-resolvent-large-argument-1} and \eqref{eq：exp for  reso-0-1} yields that for $d>3$,
		\begin{equation}\label{eq:expansion for free-4}
			\begin{split}
				R^{\pm}_0(\lambda^2; x,y)=&R^{\pm}_0(\lambda^2; x,y)\eta(\lambda|x-y|)+R^{\pm}_0(\lambda^2; x,y)\widetilde{\eta}(\lambda|x-y|) \\
				=&c_d|x-y|^{2-d}-c_d|x-y|^{2-d}\widetilde{\eta}(\lambda|x-y|)+r^{\pm}_d(\lambda,|x-y|)\eta(\lambda|x-y|) \\
				+&\frac{\lambda^{\frac{d-3}{2}}}{|x-y|^{\frac{d-1}{2}}}e^{\pm i\lambda|x-y|}\Phi^\pm(\lambda|x-y|)\widetilde{\eta}(\lambda|x-y|). 	
			\end{split}
		\end{equation}
		Here by \eqref{eq: est r_d>3}, we have the following estimate
		\begin{equation}\label{eq:est for rd with cut}
			\Big|\partial_{\lambda}^{k}\left(r^{\pm}_d(\lambda,|x-y|)\eta(\lambda|x-y|)\right) \Big| \les |x-y|^{3-d}\lambda^{1-k}, \quad  k\in \N_0.\end{equation}
		
		Finally, for all $d\ge 1$, we have 
		\begin{equation}\label{eq:expansion for free-5}
			R^{+}_0(\lambda^2; x,y)-R^{-}_0(\lambda^2; x,y)=\lambda^{\frac{d-3}{2}}e^{ i\lambda|x-y|}\frac{J^{+}(\lambda|x-y|)}{|x-y|^{\frac{d-1}{2}}}- \lambda^{\frac{d-3}{2}}e^{-i\lambda|x-y|}\frac{J^{-}(\lambda|x-y|)}{|x-y|^{\frac{d-1}{2}}},
		\end{equation}
		where $J^{\pm}(\lambda|x-y|)$ satisfies \eqref{eq-resolvent-large-argument-2}.

		\subsection{Admissible kernel and Fourier multipliers}
		
		We note that a large class of terms appearing in  the resolvent expansion have integral kernels satisfying the following admissibility condition: An operator $T$ with kernel $T(x, y)$ is called \emph{admissible} if
		\begin{equation}\label{eq-admissible condition}
			\sup_{x\in\mathbb{R}^d}\int |T(x,y)|\d y+\sup_{y\in\mathbb{R}^d}\int|T(x,y)|\d x<\infty.
		\end{equation}

		It is a standard result that any operator with an admissible kernel is bounded on $L^p(\mathbb{R}^n)$ for all $1 \leq p \leq \infty$. 
		The following lemmas are frequently used in the proof of Theorem \ref{thm-main result-rank one-d>3}--\ref{thm-main result-finite rank}.

		\begin{lemma}[{\cite[Lemma 3.8]{GV06}}]\label{lemma-GV} Let $d\ge1.$ There exists an absolute constant $C>0$ such that
			\begin{equation}\label{eq-estimates-GV-1-3-24}
				\int_{\mathbb{R}^d}|x-y|^{-k}\langle y\rangle^{-l}\d y\le C\langle x\rangle ^{-\min\{k,k+l-d\}},
			\end{equation}
			provided $l\ge 0,0\le k<d$ and $k+l>d.$ As a consequence, we also have
			\begin{equation}\label{eq-estimates-GV-1}
				\int_{\mathbb{R}^d}\int_{\mathbb{R}^d}|x-y|^{-\sigma}\langle x-\tau\rangle^{-\beta_1}\langle y\rangle^{-\beta_2}\d x\d y\le C\langle \tau\rangle^{-\sigma},
			\end{equation}
			provided $\beta_1>d,\beta_2>d$ and $0\le\sigma<d.$
		\end{lemma}
		
		Next, we give a brief overview of $L^p$-Fourier multipliers, which will be occasionally required in our analysis. For $u \in L^\infty(\mathbb{R}^n)$, define $T_u f = \mathcal{F}^{-1}(u\widehat{f})$ on $C_c^\infty(\mathbb{R}^n)$. The function $u$ is an $L^p$-Fourier multiplier if $T_u$ extends to a bounded operator on $L^p(\mathbb{R}^n)$. Denote by $M_p(\mathbb{R}^n)$ ($1 \leq p \leq \infty$) the space of such multipliers, equipped with the operator norm of $T_u$. Key properties of $M_p(\mathbb{R}^n)$ include:
		\begin{align*}
			M_p(\mathbb{R}^n) &= M_{p'}(\mathbb{R}^n) \quad (1/p + 1/p' = 1,\quad 1< p<\infty), \\
			M_p(\mathbb{R}^n) &\subset M_q(\mathbb{R}^n) \quad (1 \leq p \leq q \leq 2). 
		\end{align*}
		For $1 < p < \infty$, the Marcinkiewicz and H\"{o}rmander--Mihlin theorems characterize $M_p$-multipliers, but fail for $p=1$. The following result provides sufficient conditions for $M_1(\mathbb{R}^n)$ by using Bernstein's theorem (see e.g. in \cite{Sj-1970}). We also denote by $\mathcal{B}(X)$ the space of all bounded linear operators on $X$.
		\begin{lemma}[{\cite[Lemma 2.1]{DDY}}]\label{lem: Fourier Multiplier by Bernstein}
			Let $L \in \mathbb{N}$ with $L > \frac{d}{2}$. Assume that a measurable function $u$ satisfies one of the following conditions:\\
			$(a)$ $ u \in C^L(\mathbb{R}^d) $ and there exists a constant $ A_1 > 0 $ such that for all $ |\gamma| \leq L $
			$$
			\left|\partial^\gamma u(\xi)\right| \leq A_1 (1 + |\xi|)^{-\varepsilon_1 - |\gamma|}, \quad \xi \in \mathbb{R}^d.
			$$
			$(b)$ $ u \in C^L(\mathbb{R}^d \setminus \{0\}) $ and there exist constants $ \varepsilon_2, \varepsilon_3 > 0 $ such that for all $ |\gamma| \leq L $
			\begin{equation}\label{eq:condition for FM}
				\left| \partial^\gamma u(\xi)\right|  \leq A_2 \min\left\{|\xi|^{\varepsilon_2 - |\gamma|}, \,\,|\xi|^{-\varepsilon_3 - |\gamma|}\right\}, \quad \xi \in \mathbb{R}^d \setminus \{0\}.
			\end{equation}
			Then $u \in M_p(\mathbb{R}^d)$ for all $1\le p\le \infty$, and the norm $\|u\|_{M_p}$ depends only on  $d$, $L$, $A_i$, $p$ and $\varepsilon_i$ $(i = 1, 2, 3)$.
		\end{lemma}
		
		We establish a Fourier transform estimate that will be frequently employed in the low-energy analysis.
		\begin{lemma}\label{oscillatory estimates}
			Let $b>-1$ and $k>b+1$, and assume that $\psi(\lambda)\in C^k(\mathbb{R})$ such that
			\begin{equation}\label{symbol-est}
				\Big|\frac{d^l}{d\lambda^l} \psi(\lambda)\Big|\lesssim \lambda^{b-l}, \quad l=0,1,\cdots,k.
			\end{equation}
			Then, we have  
			\begin{equation}\label{eq.oscillatory est}
				\left| \int_0^\infty e^{i\lambda x}\psi(\lambda)\chi(\lambda)\d\lambda\right| \lesssim \langle x\rangle^{-(b+1)},\quad x\in\mathbb{R},
			\end{equation}
			where $\chi(\lambda)$ is a smooth function such that $\chi\equiv1$ on $(-\lambda_0/2,\,\lambda_0/2)$ and $\text{supp}\chi\subset[-\lambda_0,\,\lambda_0]$ for some fixed $\lambda_0>0$.
		\end{lemma}
		\begin{proof}
			We denote the integral in \eqref{eq.oscillatory est} by $I(x)$. 
			If $|x|\le 1$, \eqref{eq.oscillatory est} can be obtained by the simple fact $|I(x)|\lesssim 1$. If $|x|>1$, we split $I(x)=I_0(x)+I_1(x)$, where
			\begin{equation*}
				I_i(x)=\int_0^\infty e^{i\lambda x}\psi(\lambda)\chi(\lambda)\phi_i(\lambda|x|)\d\lambda, \quad \text{for } i=0,1,
			\end{equation*}
			where $\phi_0 \in C^{\infty}(\R)$ such that $\phi_0(\lambda)=1$ when $|\lambda|\leq 1/2$ and $\phi_0(\lambda)=0$ when $|\lambda|\geq 1$, and $\phi_1(\lambda)=1-\phi_0$.
			Thus, a direct computation shows that 
			$|I_0(x)|\lesssim |x|^{-(b+1)}$. Next, by applying integration by parts $k$ times to $I_1(x)$, and then we have
			\begin{equation*}
				|I_1(x)|\lesssim |x|^{-k}\int_{\frac{1}{2}|x|^{-1}}^\infty \lambda^{b-k}\d\lambda\lesssim |x|^{-(b+1)}.
			\end{equation*}
			This completes the proof.
		\end{proof}
		
		\section{Proof of Theorem \ref{thm-main result-rank one-d>3}}\label{sec3}
		\setcounter{equation}{0}
		In this section, we explore the $L^p$ boundedness for rank one perturbations  in dimensions $d\ge3$. 
		To start, we fix a small $0<\lambda_0<1$ and choose  a smooth cut-off function $\chi(\lambda)$ such that
		\begin{equation}\label{eq-cut-off function}
			\chi(\lambda)=1,\ \quad 0<\lambda<\frac{\lambda_0}{2};\ \ \ \ \mbox{and} \ \ \  \chi(\lambda)=0,\quad\ \lambda>\lambda_0.
		\end{equation}
		Let $\widetilde{\chi}(\lambda)=1-\chi(\lambda)$.
		We denote
		\begin{equation}\label{eq-G-0504}
			G_+^{\alpha}(\lambda):=\displaystyle \frac{\alpha}{1+\alpha F^{+}(\lambda^2)}.
		\end{equation}
		Using the above cut-off function allows us to split the second term on the right-hand side of \eqref{eq-stationary representation-V} into the low energy component
		\begin{equation}\label{eq-low energy wave operator}
			\mathcal{W}_-^l(H_\alpha,H_0)=\int_0^\infty\lambda G_+^{\alpha}(\lambda)\chi(\lambda) R_0^+(\lambda^2)\varphi\langle (R_0^+(\lambda^2)-R_0^-(\lambda^2))\cdot,\varphi\rangle \d\lambda,
		\end{equation}
		and the high energy component
		\begin{equation}\label{eq-high energy wave operator}
			\mathcal{W}_-^h(H_\alpha,H_0)=\int_0^\infty \lambda G_+^{\alpha}(\lambda)\widetilde{\chi}(\lambda) R_0^+(\lambda^2)\varphi\langle (R_0^+(\lambda^2)-R_0^-(\lambda^2))\cdot,\varphi\rangle \d\lambda.
		\end{equation}
		The integral kernels of $\mathcal{W}_-^l(H_\alpha,H_0)$ and $\mathcal{W}_-^h(H_\alpha,H_0)$ are given, respectively, by   
		\begin{equation}\label{eq-low energy wave operator-1}
			\begin{split}
				\frac{1}{\pi i}\int_0^\infty\lambda G_+^{\alpha}(\lambda){\chi}(\lambda) \Big(\int_{\mathbb{R}^{2d}}R_0^+(\lambda^2;x, x_1)(R_0^+-R_0^-)(\lambda^2;x_2,y)\varphi(x_1)\varphi(x_2)\d x_1\d x_2\Big) \d\lambda,
			\end{split}
		\end{equation}
		and
		\begin{equation}\label{eq-high energy wave operator-2}
			\begin{split}
				\frac{1}{\pi i} \int_0^\infty\lambda G_+^{\alpha}(\lambda)\widetilde{\chi}(\lambda)\Big(\int_{\mathbb{R}^{2d}}R_0^+(\lambda^2;x,x_1)(R_0^+-R_0^-)(\lambda^2;x_2,y)\varphi(x_1)\varphi(x_2)\d x_1\d x_2\Big) \d\lambda.
			\end{split}
		\end{equation}
		The remainder of this section focuses on analyzing the pointwise behavior of these kernels. 
		Theorem \ref{thm-main result-rank one-d>3} follows immediately from the boundedness of $\mathcal{W}_-^l$ and $\mathcal{W}_-^h$ established in Propositions \ref{pro-low energy-d>3} and \ref{pro-high energy-d>3}.

		\subsection{The low energy contribution}
		In this subsection, we aim to prove Proposition \ref{pro-low energy-d>3} below. First, we establish asymptotic  properties of $F^\pm(\lambda^2)$ and $G_+^{\alpha}(\lambda)$ as $\lambda\to0$. Similar results have been established in \cite{CHZ23}.

		\begin{lemma}\label{lemma-F low energy-d>3}
			Let $d\ge3$ and assume  \eqref{eq-decay condition} holds. For $0<\lambda<\frac12$, we have
			\begin{equation}\label{eq-F expansion-d>3}
				F^\pm(\lambda^2)=a_0+r_d^\pm(\lambda),
			\end{equation}
			where 
			\begin{equation}\label{eq-F expansion-d>3-01}
				a_0=c_d\left\langle \int_{\mathbb{R}^d}|x-y|^{2-d}\varphi(y)\d y,\varphi(x)\right\rangle\neq 0.
			\end{equation}
			Moreover, the remainder terms $r_d^\pm(\lambda)$ satisfy
			\begin{equation}\label{eq: est for r_d-0}
				\Big|\frac{\d^k}{\d\lambda^k}r_d^\pm(\lambda)\Big|\lesssim \, \lambda^{1-k},\quad k=0,1,\cdots,\left[\frac{d+3}{2}\right].
			\end{equation}
		\end{lemma}
		\begin{proof}
			The expansion \eqref{eq-F expansion-d>3} and the estimate \eqref{eq: est for r_d-0} follow immediately from \eqref{eq：exp for  reso-0-1},  \eqref{eq: est r_d>3}, together with the decay assumption \eqref{eq-decay condition}. Moreover, since $c_d|x|^{2-d}$ is the fundamental solution for $-\Delta$, we have  $a_0=\langle(-\Delta)^{-1}\varphi,\, \varphi\rangle \neq 0$. 
		\end{proof}

		By Lemma \ref{lemma-F low energy-d>3} and Taylor expansion, we obtain the following asymptotic expansions for $G^{\alpha}_+(\lambda)\chi(\lambda)$:
		\begin{lemma}\label{lm3.2}
			Let $d\ge3$ and assume  \eqref{eq-decay condition} holds. Then there exists a small $\lambda_0>0$ such that 
			\begin{equation}\label{eq-G-d>3}
				G^{\alpha}_+(\lambda)=\frac{\alpha}{1+\alpha a_0}+\widetilde{r}_d(\lambda), \qquad   0<\lambda<\lambda_0,
			\end{equation}
			where $a_0$ is defined in \eqref{eq-F expansion-d>3-01}, and the remainder term $\widetilde{r}_d(\lambda)$ satisfies
			\begin{equation}\label{eq-r-d odd-1}
				\bigg|\frac{\d^k }{\d\lambda^k}\widetilde{r}_d(\lambda)\bigg|\lesssim \lambda^{1-k},\quad \quad k=0,1,\cdots, \left[\frac{d+3}{2}\right].
			\end{equation}
		\end{lemma}

		\begin{proposition}\label{pro-low energy-d>3}
			Fix $\lambda_0$ as in Lemma \ref{lm3.2}. Let $d \geq 3$ and  the decay assumption \eqref{eq-decay condition} holds. Then the low-energy component $\mathcal{W}_-^l(H_\alpha,H_0)$ defined in \eqref{eq-low energy wave operator} is bounded on $L^p(\mathbb{R}^d)$ for all $1 \leq p \leq \infty$.
		\end{proposition}
		
		\begin{proof}
			Let $\mathcal{W}_-^l(x, y)$ denote the integral kernel of $\mathcal{W}_-^l(H_\alpha,H_0)$, by \eqref{eq-stationary representation-one rank}, we have
			\begin{equation}\label{eq-ker-0703-1}
				\mathcal{W}_-^l(x, y)=	\frac{1}{\pi i}\int_0^\infty\lambda G_+^{\alpha}(\lambda){\chi}(\lambda) \Big(\int_{\mathbb{R}^{2d}}R_0^+(\lambda^2;x, x_1)(R_0^+-R_0^-)(\lambda^2;x_2,y)\varphi(x_1)\varphi(x_2)\d x_1\d x_2\Big) \d\lambda.
			\end{equation}
			We divide the proof into three steps.
			
			\textbf{Step 1: The $L^p-$boundedness for $d=3$ ($1\le p\le \infty$).}
			
			By \eqref{eq2.3}, the  integral of  \eqref{eq-ker-0703-1}  can be written as 
			$$ \frac{1}{16\pi^3 i}\int_{\mathbb{R}^{6}}\int_0^\infty \left( e^{ i\lambda\rho_+}-e^{ i\lambda\rho_-}\right) \lambda G_+^{\alpha}(\lambda){\chi}(\lambda)\d\lambda \frac{\varphi(x_1)\varphi(x_2)}{|x-x_1||x_2-y|}\d x_1\d x_2 \d\lambda,$$
			where $\rho_\pm=|x-x_1|\pm|x_2-y|$. Applying the expansion \eqref{eq-G-d>3} yields the decomposition:
			$$
			\mathcal{W}_-^l(x, y)=\Pi_0(x,y)+ \Pi_1^+(x,y)-\Pi_1^-(x,y),
			$$
			where 
			\begin{equation}\label{eq-Pi-0-0703}
				\Pi_0(x,y):=\frac{\alpha}{16\pi^3 ({1+\alpha a_0})i}\int_{\mathbb{R}^{6}}\int_0^\infty \left( e^{ i\lambda\rho_+}-e^{ i\lambda\rho_-}\right)\lambda \chi(\lambda)\d\lambda\, \frac{\varphi(x_1)\varphi(x_2)}{|x-x_1||x_2-y|} \d x_1\d x_2, 	
			\end{equation}
			and  
			\begin{equation}\label{eq-Pi-0-0703-2}
				\Pi_1^\pm(x,y):=\frac{\alpha}{16\pi^3 ({1+\alpha a_0})i}\int_{\mathbb{R}^{6}}\int_0^\infty e^{ i\lambda\rho_\pm} \lambda \widetilde{r}_d(\lambda)\d\lambda\, \frac{\varphi(x_1)\varphi(x_2)}{|x-x_1||x_2-y|} \d x_1\d x_2. 
			\end{equation}
			
			For $\Pi_0(x,y)$, we  decompose the kernel as 
			\begin{equation*}
				\Pi_0(x,y):= \sum_{j=1}^3 \Pi_{0,j}(x,y),
			\end{equation*}
			where each $\Pi_{0,j}$ is defined as the restriction of $\Pi_0$ to a distinct region:
			\begin{itemize}
				\item  $\Pi_{0,1}$ corresponds to the region $ |\rho_-|\le |\rho_+|\le 1$ in the integrand  \eqref{eq-Pi-0-0703};
				\item  $\Pi_{0,2}$ corresponds to the region $|\rho_-|\le 1\le |\rho$ in the integrand  \eqref{eq-Pi-0-0703};
				\item  $\Pi_{0,3}$ corresponds to the region $1<|\rho_-|\le |\rho_+|$ in the integrand  \eqref{eq-Pi-0-0703}.
			\end{itemize}
			
			In the region  $|\rho_-|\le |\rho_+|\le 1$, we immediately obtain that 
			$$ \left|\Pi_{0,1}(x,y) \right|\les \sum_{\pm}\int_{\mathbb{R}^{6}}  \frac{|\varphi(x_1)||\varphi(x_2)|}{|x-x_1||x_2-y|\langle \rho_\pm\rangle^3} \d x_1\d x_2. $$
			Since $\varphi \in L^1(\mathbb{R}^3)$ by \eqref{eq-decay condition}, applying Fubini's theorem yields
			\begin{equation}\label{eq-qdmi-7-4-1}
				\begin{split}
					\left\|\Pi_{0,1}(x,y)\right\|_{L^1_x} 
					&\lesssim  \sum_{\pm}\int_{\mathbb{R}^3} \int_{\mathbb{R}^6} \frac{|\varphi(x_1)||\varphi(x_2)|}{|x-x_1||x_2-y|\langle \rho_\pm\rangle^3} \d x_1 \d x_2 \d x \\
					&\lesssim  \sum_{\pm}\int_{\mathbb{R}^6} \frac{|\varphi(x_2)|}{|x_1||x_2-y|\langle |x_1|\pm|x_2-y|\rangle^3} \d x_1 \d x_2 \\
					&\lesssim  \sum_{\pm}\int_0^\infty \int_{\mathbb{R}^3} \frac{r|\varphi(x_2)|}{|x_2-y|\langle r\pm|x_2-y|\rangle^3} \d r \d x_2 \\
					&\lesssim \int_{\mathbb{R}^3} \frac{\langle x_2-y\rangle|\varphi(x_2)|}{|x_2-y|} \d x_2 \\
					&\lesssim 1,
				\end{split}
			\end{equation}
			By symmetry, we also have 
			$\sup_{x}	\left\|\Pi_{0,1}(x,y)\right\|_{L^1_y} \lesssim 1.$ Thus $\Pi_{0,1}(x,y)$ is admissible. 
			
			In the second region where $|\rho_-| \leq 1 \leq |\rho_+|$, the integral involving $\rho_-$ can be treated using the same method as in the case $|\rho_-| \leq |\rho_+| \leq 1$.
			Thus, we apply integration by parts twice in the $\lambda$ integral (with respect to $e^{i\lambda\rho_{+}}$ ) and obtain
			\begin{align*}
				\left|\Pi_{0,2}(x,y) \right|\les& \int_{\mathbb{R}^{6}}  \frac{|\varphi(x_1)||\varphi(x_2)|}{|x-x_1||x_2-y| \langle \rho_+\rangle^{2}} +\frac{|\varphi(x_1)||\varphi(x_2)|}{|x-x_1||x_2-y|\langle \rho_{-}\rangle^3}  \d x_1\d x_2 \\
				\les& \sum_{\pm}\int_{\mathbb{R}^{6}}  \frac{|\varphi(x_1)||\varphi(x_2)|}{|x-x_1||x_2-y|\langle \rho_\pm\rangle^3}  \d x_1\d x_2,
			\end{align*}
			where we have used the inequality
			$$ \langle\rho_+\rangle^{-2} \le \langle\rho_+\rangle^{-2}\langle\rho_-\rangle^{-1}\les \langle\rho_+\rangle^{-3}+\langle\rho_-\rangle^{-3}, $$
			when $|\rho_-|\le 1\le |\rho_+|$. 
			Thus, in view of \eqref{eq-qdmi-7-4-1},  $\Pi_{0,2}(x,y)$ is also admissible.
			
			In the third region where $1<|\rho_-|\le |\rho_+|$, we perform integration by parts twice in the $\lambda$ integral (with respect to $e^{i\lambda\rho_{\pm}}$ ) to get 
			$$\int_0^\infty \left( e^{ i\lambda\rho_+}-e^{ i\lambda\rho_-}\right) \lambda G_+^{\alpha}(\lambda){\chi}(\lambda)\d\lambda=\rho_+^{-2}-\rho_-^{-2}
			-\int_0^\infty \left( \frac{e^{ i\lambda\rho_+}}{\rho_+^2}-\frac{e^{ i\lambda\rho_-}}{\rho_-^2}\right)\big(2\chi{'}(\lambda)+\lambda \chi{''}(\lambda)\big)\d\lambda.$$
			Since the support of $2\chi{'}(\lambda)+\lambda \chi{''}(\lambda)$ is away from zero,   performing one more integration by parts yields 
			\begin{align*}
				\left|\Pi_{0,3}(x,y) \right|\les& \int_{\mathbb{R}^{6}}  \frac{|\varphi(x_1)||\varphi(x_2)|}{|\rho_+|^2 |\rho_-|^2} \d x_1\d x_2+\sum_{\pm}\int_{\mathbb{R}^{6}}  \frac{|\varphi(x_1)||\varphi(x_2)|}{|x-x_1||x_2-y|\langle \rho_\pm\rangle^3} \d x_1\d x_2 \\
				\les& \sum_{\pm} \int_{\mathbb{R}^{6}}  \frac{|\varphi(x_1)||\varphi(x_2)|}{\langle\rho_\pm\rangle^4} \d x_1\d x_2
				+\sum_{\pm}\int_{\mathbb{R}^{6}}  \frac{|\varphi(x_1)||\varphi(x_2)|}{|x-x_1||x_2-y|\langle \rho_\pm\rangle^3} \d x_1\d x_2.
			\end{align*}
			This implies that $\Pi_{0,3}(x,y)$ is admissible.

			For $\Pi_1^\pm(x,y)$, we first analyze the inner $\lambda$-integral  in \eqref{eq-Pi-0-0703-2}, that is
			\begin{align}\label{eq:lambda-int d=3}
				\Lambda =\int_0^\infty e^{ i\lambda\rho_\pm}  \lambda \widetilde{r}_d(\lambda)\chi(\lambda) \d\lambda.
			\end{align} 
			By noting \eqref{eq-r-d odd-1} and using Lemma \ref{oscillatory estimates} with $b=\frac 32$, we obtain
			$$
			\left|\Lambda\right| \lesssim \langle \rho_\pm\rangle^{-\frac52}.
			$$
			Applying the argument from \eqref{eq-qdmi-7-4-1}, it follows that  $\Pi_1^\pm(x,y)$ is admissible.
			
			Consequently,  $\mathcal{W}_-^l(H_\alpha,H_0)\in \mathcal{B}(L^p(\R^3))$ for all $1\le p\le \infty$.

			
			\textbf{Step 2: The $L^\infty$-boundedness for $d\ge 4$.}
			
			Using the representation \eqref{eq:expansion for free-4},   $\mathcal{W}_-^l(H_\alpha,H_0)$ comprises the following four terms:
			$$
			\mathcal{W}_-^l(H_\alpha,H_0)=\sum_{j=0}^4I_j,
			$$
			where the integral kernel of each $I_j$ is defined as
			\begin{align*}
				I_1(x,y)=\frac{c_d}{\pi i}\int_0^\infty&\lambda G_+^{\alpha}(\lambda){\chi}(\lambda) \int_{\mathbb{R}^{2d}}|x-x_1|^{2-d}(R_0^+-R_0^-)(\lambda^2;x_2,y)\varphi(x_1)\varphi(x_2)\d x_1\d x_2\d\lambda, \\
				I_2(x,y)=\frac{c_d}{\pi i}\int_0^\infty&\lambda G_+^{\alpha}(\lambda){\chi}(\lambda) \int_{\mathbb{R}^{2d}}|x-x_1|^{2-d}\widetilde{\eta}(\lambda|x-x_1|)\, (R_0^+-R_0^-)(\lambda^2;x_2,y)\varphi(x_1)\varphi(x_2)\d x_1\d x_2\d\lambda, \\
				I_3(x,y) =\frac{1}{\pi i}\int_0^\infty&\lambda G_+^{\alpha}(\lambda){\chi}(\lambda) \\
				\times& \int_{\mathbb{R}^{2d}} r^{+}_d(\lambda,|x-x_1|)\eta(\lambda|x-x_1|)\, (R_0^+-R_0^-)(\lambda^2;x_2,y)\varphi(x_1)\varphi(x_2)\d x_1\d x_2 \d\lambda,
			\end{align*}
			and 
			\begin{align*}
				I_4(x,y)=\frac{1}{\pi i}\int_0^\infty &\lambda^{\frac{d-1}{2}} G_+^{\alpha}(\lambda){\chi}(\lambda)  \\
				\times& \int_{\mathbb{R}^{2d}} e^{ i\lambda|x-x_1|}\frac{\Phi^+(\lambda|x-x_1|)}{|x-x_1|^{\frac{d-1}{2}}}\widetilde{\eta}(\lambda|x-x_1|)\, (R_0^+-R_0^-)(\lambda^2;x_2,y)\varphi(x_1)\varphi(x_2)\d x_1\d x_2 \d\lambda.
			\end{align*}
			
			Observe that $\frac{\lambda}{\pi i}(R_0^+-R_0^-)(\lambda^2)=dE_{\lambda}(\sqrt{-\Delta})$ is the spectral measure of the $\sqrt{-\Delta}$. By spectral theorem, we can
			rewrite  $I_1$ in the form
			$$ I_1 f(x)=c_d \int_{\mathbb{R}^{d}}|x-x_1|^{2-d}\varphi(x_1) \d x_1 \cdot \int_{\mathbb{R}^{d}} \varphi(y) \left(G_+^{\alpha}(\sqrt{-\Delta})\chi(\sqrt{-\Delta})f\right)(y) \d y. $$
			By \eqref{eq-G-d>3}, the symbol of $G_+^{\alpha}(\sqrt{-\Delta})\chi(\sqrt{-\Delta})$ is 
			$$ \frac{\alpha}{1+\alpha a_0}\chi(|\xi|)+\widetilde{r}_d(|\xi|)\chi(|\xi|).$$
			By \eqref{eq-r-d odd-1}, $\widetilde{r}_d(|\xi|)\chi(|\xi|)$ satisfies the point-wise estimate \eqref{eq:condition for FM}, thus $\widetilde{r}_d(\sqrt{-\Delta})\chi(\sqrt{-\Delta})\in \mathcal{B}(L^\infty)$ by Lemma \ref{lem: Fourier Multiplier by Bernstein}. On the other hand, it is clear that $\chi(\sqrt{-\Delta})\in\mathcal{B}(L^\infty)$. Thus it follows that 
			$$G_+^{\alpha}(\sqrt{-\Delta})\chi(\sqrt{-\Delta})\in \mathcal{B}(L^\infty).$$
			Meanwhile, by the decay assumption \eqref{eq-decay condition} holds, we derive from \eqref{eq-estimates-GV-1-3-24} that 
			$$\sup_{x\in \R^d}\int_{\mathbb{R}^{d}}|x-x_1|^{2-d}\varphi(x_1) \d x_1 \lesssim 1.$$
			Therefore $I_1 \in \mathcal{B}(L^\infty)$.
			
			Similarly,  the operators $I_2$ and $I_3$ can be written as 
			$$ I_2f(x)=c_d \int_{\mathbb{R}^{d}}|x-x_1|^{2-d}\d x_1    \,\int_{\mathbb{R}^{d}}\varphi(y)  \left(\widetilde{\eta}(|x-x_1|\sqrt{-\Delta})\varphi(x_1) G_+^{\alpha}(\sqrt{-\Delta})\chi(\sqrt{-\Delta})f\right)(y) \d y $$
			and 
			$$ I_3f(x)=c_d \int_{\mathbb{R}^{d}}\varphi(x_1) \d x_1    \,\int_{\mathbb{R}^{d}}\varphi(y)  \left(r^{+}_d(\sqrt{-\Delta},|x-x_1|)\eta(|x-x_1|\sqrt{-\Delta})G_+^{\alpha}(\sqrt{-\Delta})\chi(\sqrt{-\Delta})f\right)(y) \d y $$
			respectively.
			By \eqref{eq-G-d>3},  \eqref{eq:est for rd with cut}, and the fact that $\varphi\in L^1$,  it follows that
			$$ \sup_{x\in \R^d}\left| \partial_\xi^\gamma \left(\int_{\mathbb{R}^{d}}|x-x_1|^{2-d}\widetilde{\eta}(|\xi||x-x_1|)\varphi(x_1) \d x_1    \, G_+^{\alpha}(|\xi|)\chi(|\xi|) \right) \right| \lesssim |\xi|^{d-2-|\gamma|},    \quad \xi \in \mathbb{R}^d \setminus \{0\},
			$$
			and
			$$ \sup_{x\in \R^d}\left| \partial_\xi^\gamma \left(\int_{\mathbb{R}^{d}}r^{\pm}_d(|\xi|,|x-x_1|)\eta( |\xi||x-x_1|)\varphi(x_1) \d x_1    \, G_+^{\alpha}(|\xi|)\chi(|\xi|) \right) \right| \lesssim |\xi|^{1-|\gamma|},    \quad \xi \in \mathbb{R}^d \setminus \{0\},$$
			hold for all $ |\gamma| \leq [\frac{d+3}{2}]$.
			Thus by Lemma \ref{lem: Fourier Multiplier by Bernstein},  $I_2, I_3 \in \mathcal{B}(L^\infty)$.
			
			For the term $I_4(x,y)$, we use the expansion in \eqref{eq:expansion for free-5} to rewrite it as a difference of
			\begin{align}\label{eq: expression of I4}
				\frac{1}{\pi i} \int_{\mathbb{R}^{2d}}\int_0^\infty e^{ i\lambda\rho_\pm}   G_+^{\alpha}(\lambda) \chi(\lambda)\lambda^{d-2} \Psi^\pm(\lambda,x,x_1,x_2,y)   \d\lambda \, \frac{\varphi(x_1)\varphi(x_2)}{|x-x_1|^{\frac{d-1}{2}}|x_2-y|^{\frac{d-1}{2}}} \d x_1\d x_2,
			\end{align}
			where  $\rho_\pm=|x-x_1|\pm|x_2-y|$, and
			$$\Psi^\pm(\lambda,x,x_1,x_2,y):=\Phi^+(\lambda|x-x_1|)J^{\pm}(\lambda|x_2-y|)\widetilde{\eta}(\lambda|x-x_1|).$$
			Similar to  estimating \eqref{eq-Pi-0-0703-2} , we study the inner $\lambda$-integral in \eqref{eq: expression of I4}, that is 
			\begin{align}\label{eq:lambda-int}
				\tilde{\Lambda} &=\int_0^\infty e^{ i\lambda\rho_\pm} \lambda^{d-2}   G_+^{\alpha}(\lambda) \chi(\lambda) \Psi^\pm(\lambda,x,x_1,x_2,y)  
			\end{align} 
			By \eqref{eq-resolvent-large-argument-2}, we have for $k=0,1,2,\ldots,\left[\frac{d+3}{2}\right]$, 
			$$ \sup_{x,x_1,x_2,y}\Big|\partial_{\lambda}^k\big[ G_+^{\alpha}(\lambda) \chi(\lambda) \Psi^\pm(\lambda,x,x_1,x_2,y)\big] \Big|\lesssim \lambda^{-k}, \quad 0<\lambda<\lambda_0. $$
			Note that $d-2-[\frac{d+3}{2}]>-\frac 54$ when $d\ge 4$. By using Lemma \ref{oscillatory estimates} with $[\frac{d+3}{2}]-\frac 54$, it follows that 
			\begin{equation}\label{eq-Lambda-0705}
				\tilde{\Lambda}\lesssim  \left\langle \rho_\pm\right\rangle^{\frac14-[\frac{d+3}{2}]}.
			\end{equation}
			Thus we have
			\begin{equation}\label{eq-I-4-0705}
				\begin{split}
					\sup_{x\in \R^d}\big\|I_{4}(x,y) \big\|_{L^1_y} \lesssim& 
					\sup_{x\in \R^d}\int_{\mathbb{R}^{d}} \int_{\mathbb{R}^{2d}} \big\langle \rho_\pm\big\rangle^{\frac14-\left[\frac{d+3}{2}\right]}\frac{|\varphi(x_1)||\varphi(x_2)|}{|x-x_1|^{\frac{d-1}{2}}|x_2-y|^{\frac{d-1}{2}}} \d x_1\d x_2 \d y \\
					\lesssim & \sup_{x\in \R^d}\int_{\mathbb{R}^{d}}\int_{0}^\infty \big\langle r\pm|x-x_1|\big\rangle^{\frac14-\left[\frac{d+3}{2}\right]}\frac{ |\varphi(x_1)|}{|x-x_1|^{\frac{d-1}{2}}} r^{\frac{d-1}{2}}\d r\d x_1 \\
					\lesssim & \sup_{x\in \R^d}\int_{\mathbb{R}^{d}} \frac{ \langle x-x_1\rangle^{\frac{d-1}{2}}|\varphi(x_1)|}{|x-x_1|^{\frac{d-1}{2}}} \d x_1 
					\lesssim 1,
				\end{split}
			\end{equation} 
			where in the first inequality, we used \eqref{eq-Lambda-0705};
			in the second inequality, we first applied change of variables $x_2-y=x_2'$, noting that $\int_{\R^d}{|\varphi(x_2'+y)|\d y }<M<\infty$ (independent of $x_2'$), employed polar coordinates by setting $|x_2'|=r>0$;
			while in the third inequality, we used the fact that $\frac{d-1}{2}+\frac14-\left[\frac{d+3}{2}\right]<-1$. Consequently, this yields that $I_{4} \in \mathcal{B}(L^\infty)$. 
			
			Collecting these results for $I_1$ through $I_4$, we conclude that 
			$\mathcal{W}_-^l(H_\alpha,H_0)\in \mathcal{B}(L^\infty)$.
			
			\textbf{Step 3: The $L^1$-boundedness for $d\ge 4$.}
			
			Using the expressions \eqref{eq-free kernel-1} and \eqref{eq:expansion for free-5}, we write 
			$$
			\mathcal{W}_-^l(H_\alpha,H_0)=J_0^{+}+J_1^{+}-J_0^{-}-J_1^{-},
			$$
			where the integral  kernel of $J_s^{\pm}$ ($s=0, 1$) is given by 
			\begin{align*}\label{eq: expression for L^1 bound}
				J_s^{\pm}(x,y):=\int_{\mathbb{R}^{2d}}\int_0^\infty e^{ i\lambda\rho_\pm}   G_+^{\alpha}(\lambda) \chi(\lambda)\lambda^{d-2-\frac{d-3}{2}s} \, \Phi_s^\pm(\lambda,x,x_1,x_2,y)   \d\lambda \, \frac{\varphi(x_1)\varphi(x_2)}{|x-x_1|^{\frac{d-1}{2}+\frac{d-3}{2}s}|x_2-y|^{\frac{d-1}{2}}} \d x_1\d x_2
			\end{align*}
			with
			$$ \Phi_s^\pm(\lambda,x,x_1,x_2,y)=w_s^+(\lambda|x-x_1|)J^{\pm}(\lambda|x_2-y|).$$
			Following the method for estimating $I_4$ above, we derive that for $s=0, 1$,
			\begin{align*}
				\left| J_s^{\pm}(x,y)\right|\lesssim \int_{\mathbb{R}^{2d}} \big\langle \rho_\pm\big\rangle^{\frac14+\frac{d-3}{2}s-\left[\frac{d+3}{2}\right]}\frac{|\varphi(x_1)||\varphi(x_2)|}{|x-x_1|^{\frac{d-1}{2}+\frac{d-3}{2}s}|x_2-y|^{\frac{d-1}{2}}} \d x_1\d x_2,  
			\end{align*}
			which, in turn, implies that
			$$
			\sup_{y\in\R^d}\int_{\R^d}|J_s^{\pm}(x,y)|\d x<\infty,\qquad s=0,1.
			$$
			Consequently, we have $\mathcal{W}_-^l(H_\alpha,H_0)\in \mathcal{B}(L^1)$.
			
			Therefore  the proof of Proposition \ref{pro-low energy-d>3} is complete. 
		\end{proof}
		
		\subsection{The high energy contribution}
		In this subsection, we investigate the $L^p$-boundedness of the high-energy component $\mathcal{W}_-^h(H_\alpha,H_0)$.  We begin with the following estimate for $G_+^{\alpha}(\lambda)$ when $\lambda>\lambda_0$.
		
		\begin{lemma}\label{lem:est for high G}
			Let $d\ge1$ and assume condition \eqref{eq-decay condition} holds. For $\lambda>\lambda_0$ with any fixed $\lambda_0>0$,   $G_+^{\alpha}(\lambda)$ satisfies
			\begin{equation}\label{eq:est for G high-1}
				\left| \frac{d^k}{\d\lambda^k}G_+^{\alpha}(\lambda)\right| \lesssim 
				\begin{cases}
					1,& \quad k=0,\\
					\lambda^{-1}, &\quad 1\le k\le [\frac{d+3}{2}]. 
				\end{cases}
			\end{equation}
		\end{lemma}
		\begin{proof}
			The well-known limiting absorption principle states that for any $d\ge 1$ and $\lambda>\lambda_0$,
			\begin{equation}\label{eq: lim absor prin}
				\Big\|\frac{d^k}{d\lambda^k}R_0^{\pm}(\lambda^2)\Big\|_{L^2_\sigma(\mathbb{R}^d) \to L^2_{-\sigma}(\mathbb{R}^d)} \lesssim_{\lambda_0, k} \lambda^{-1} 
			\end{equation}
			holds for all $\lambda> \lambda_0$ and $\sigma > \frac{1}{2} + k$.
			Combining this with the definition of $G_+^{\alpha}(\lambda)$ and the decay assumption on $\varphi$ (\eqref{eq-decay condition}) yields the estimate \eqref{eq:est for G high-1}.
		\end{proof}

		\begin{proposition}\label{pro-high energy-d>3}
			Let $d\ge3$ and assume that $\varphi$ satisfies \textbf{Condition} $H_\alpha.$ For any fixed $\lambda_0>0$, the high-energy component $\mathcal{W}_-^h(H_\alpha,H_0)$ is bounded on $L^p(\mathbb{R}^d)$ for all $1\le p\le\infty.$ 
		\end{proposition}
		\begin{proof}
			Using \eqref{eq-free kernel-1} and \eqref{eq:expansion for free-5} , the kernel of $\mathcal{W}_-^h(H_\alpha,H_0)$ can be expressed as a combination of four terms
			\begin{align}\label{eq: expression for high dge3}
				\tilde{J}_s^{\pm}(x,y):=\int_{\mathbb{R}^{2d}}\int_0^\infty e^{ i\lambda\rho_\pm}   G_+^{\alpha}(\lambda) \widetilde{\chi}(\lambda)\lambda^{d-2-\frac{d-3}{2}s} \Phi_s^\pm(\lambda,x,x_1,x_2,y)   \d\lambda \, \frac{\varphi(x_1)\varphi(x_2)}{|x-x_1|^{\frac{d-1}{2}+\frac{d-3}{2}s}|x_2-y|^{\frac{d-1}{2}}} \d x_1\d x_2
			\end{align}
			for $s=0,1$, where $\rho_\pm=|x-x_1|\pm|x_2-y|$ and
			$$ \Phi_s^\pm(\lambda,x,x_1,x_2,y)=w_s^+(\lambda|x-x_1|)J^{\pm}(\lambda|x_2-y|).$$
			To establish the admissibility of the kernels $\tilde{J}_s^{\pm}(x,y)$, we take advantage of the smoothness and decay properties
			of $\varphi$ when $d\ge 3$.   These properties  enable  us to perform integration by parts with respect to the spatial variables. More precisely, we define the operators
			\begin{equation}\label{eq-differential operators}
				L_{x_1}:=\frac{1}{i\lambda}\frac{x_1-x}{|x_1-x|}\cdot\nabla_{x_1},\ \ \ \ \ \ L_{x_2}:=\frac{1}{i\lambda}\frac{x_2-y}{|x_2-y|}\cdot\nabla_{x_2},
			\end{equation}
			along with their formal adjoints
			\begin{equation}\label{eq-differential operators-dual}
				L_{x_1}^*:=\frac{i}{\lambda}\nabla_{x_1}\big(\frac{x_1-x}{|x_1-x|}),\ \ \ \ \ \ L_{x_2}^*:=\frac{i}{\lambda}\nabla_{x_2}\big(\frac{x_2-y}{|x_2-y|}\big).
			\end{equation}
			Thus we  have $L_{x_1}e^{\pm i\lambda|x-x_1|}=e^{\pm i\lambda|x-x_1|}$ and $L_{x_2}e^{i\lambda|x_2-y|}=e^{i\lambda|x_2-y|}$. Let 
			$$n_{s,d}=\begin{cases}
				[\frac{d}{2}], &\quad s=0, \\
				1, &\quad s=1.
			\end{cases}$$
			Applying the operator $L_{x_1}$  $n_{s,d}$ times to the phase factors $e^{\pm i\lambda|x-x_1|}$ and  $L_{x_2}$ $n_{0,d}$ times to $e^{i\lambda|x_2-y|}$ in $\tilde{J}_s^{\pm}(x,y)$, we have
			\begin{align*}
				\tilde{J}_s^{\pm}(x,y)=\int_{\mathbb{R}^{2d}}\int_0^\infty& e^{ i\lambda\rho_\pm}   G_+^{\alpha}(\lambda) \widetilde{\chi}(\lambda)\lambda^{d-2-\frac{d-3}{2}s} \\
				&\times \big(L_{x_1}^*\big)^{n_{s,d}}\big(L_{x_2}^*\big)^{n_{0,d}} \Phi_s^\pm(\lambda,x,x_1,x_2,y)    \frac{\varphi(x_1)\varphi(x_2)}{|x-x_1|^{\frac{d-1}{2}+\frac{d-3}{2}s}|x_2-y|^{\frac{d-1}{2}}} \d\lambda \, \d x_1\d x_2,
			\end{align*}
			which can be further written as a linear combination of 
			\begin{equation}\label{eq: high form}
				\int_{\mathbb{R}^{2d}}\int_0^\infty e^{ i\lambda\rho_\pm}   G_+^{\alpha}(\lambda) \widetilde{\chi}(2\lambda)\lambda^{d-2-\frac{d-3}{2}s-n_{s,d}-n_{0,d}} 
				{\Phi}^{L}_s(\lambda,x,x_1,x_2,y)    \frac{\nabla^{\gamma_1}_{x_1}\varphi(x_1)\nabla_{x_2}^{\gamma_2}\varphi(x_2)}{|x-x_1|^{k_1}|x_2-y|^{k_2}} \d\lambda \, \d x_1\d x_2,
			\end{equation}
			where 
			$$0\le |\gamma_1|\le n_{s,d},\ \ \ 0\le |\gamma_2|\le n_{0,d},\ \ \ \frac{d-1}{2}\le k_1,k_2\le d-1.$$ 
			By \eqref{eq-w estiamte} and the properties of Bessel functions $J^{\pm}$, it follows that
			$$\sup_{x,x_1,x_2,y}\left|\partial_{\lambda}^{k} {\Phi}^{L}_s(\lambda,x,x_1,x_2,y) \right|\lesssim \lambda^{-k}, \quad k=0,1,\ldots, \left[\frac{d+3}{2}\right].$$
			we estimate the inner $\lambda$-integral in \eqref{eq: high form} and   decompose the integral into
			\begin{align*}
				\widehat{\Lambda} &:=\int_0^\infty e^{ i\lambda\rho_\pm}   G_+^{\alpha}(\lambda) \widetilde{\chi}(2\lambda)\lambda^{d-2-\frac{d-3}{2}s-n_{s,d}-n_{0,d}} 
				{\Phi}^{L}_s(\lambda,x,x_1,x_2,y)      \d\lambda\nonumber\\
				&=\int_0^\infty e^{ i\lambda\rho_\pm}   G_+^{\alpha}(\lambda) \widetilde{\chi}(2\lambda)\lambda^{d-2-\frac{d-3}{2}s-n_{s,d}-n_{0,d}} 
				{\Phi}^{L}_s(\lambda,x,x_1,x_2,y)   \chi(|\rho_\pm|\lambda)  \d\lambda\nonumber\\
				&+\int_0^\infty e^{ i\lambda\rho_\pm}   G_+^{\alpha}(\lambda) \widetilde{\chi}(2\lambda)\lambda^{d-2-\frac{d-3}{2}s-n_{s,d}-n_{0,d}} 
				{\Phi}^{L}_s(\lambda,x,x_1,x_2,y)   (1-\chi(|\rho_\pm|\lambda))  \d\lambda\nonumber\\
				&:=\widehat{\Lambda}_1+ \widehat{\Lambda}_2
			\end{align*} 
			For the term  $\widehat{\Lambda}_1$, notice that $d-2-\frac{d-3}{2}s-n_{s,d}-n_{0,d} \leq -1$,  one has 
			\begin{align*}
				\left| \widehat{\Lambda}_1\right|
				\lesssim \int_{\frac{\lambda_0}{2}}^{|\rho_\pm|^{-1}} \lambda^{-1}   \d\lambda \lesssim \langle\log|\rho_\pm|\rangle.
			\end{align*}
			For the term $\widehat{\Lambda}_2$, we perform integration by parts $[\frac{d+3}{2}]$ times  to obtain
			\begin{align*}
				\left| \widehat{\Lambda}_2\right|
				&\lesssim\big|\rho_\pm\big|^{-\left[\frac{d+3}{2}\right]}\int_{\frac{\lambda_0}{2}}^{\infty} \lambda^{d-2-\frac{d-3}{2}s-n_{s,d}-n_{0,d}-\left[\frac{d+3}{2}\right]}   \d\lambda \\
				&\lesssim  \big|\rho_\pm\big|^{-\left[\frac{d+3}{2}\right]}.
			\end{align*}
			Substituting the above estimate for $\widehat{\Lambda}$ into \eqref{eq: high form} gives
			\[
			\left|\text{\eqref{eq: high form}}\right|\lesssim \sum_{0\le |\gamma_1|\le n_{s,d}}\sum_{0\le |\gamma_2|\le n_{0,d}}\int_{\mathbb{R}^{2d}} \frac{|\nabla^{\gamma_1}_{x_1}\varphi(x_1)\nabla_{x_2}^{\gamma_2}\varphi(x_2)| \langle\log|\rho_\pm|\rangle}{|x-x_1|^{k_1}|x_2-y|^{k_2}\langle\rho_\pm\rangle^{[\frac{d+3}{2}]}} \d\lambda \d x_1 \d x_2.
			\]
			Thanks to \eqref{eq-smoothness condition}, we can  apply the same arguments  to those establishing \eqref{eq-Lambda-0705} to deduce that, for $s=0,1$,
			$$
			\sup_{y\in\R^d}\int_{\R^d}|\tilde{J}_s^{\pm}(x,y)|\d x<\infty,\quad \mbox{and}\quad\sup_{x\in\R^d}\int_{\R^d}|\tilde{J}_s^{\pm}(x,y)|\d y<\infty.
			$$
			Therefore $\tilde{J}_s^{\pm}(x,y)$ is admissible and the proof is complete.
		\end{proof}

		\section{Proof of Theorem \ref{thm-main result-rank one-d=1,2}}\label{sec4}
		
		This section presents the proof of Theorem \ref{thm-main result-rank one-d=1,2}. Let $\mathcal{W}_-^l(H_\alpha,H_0)$ and $\mathcal{W}_-^h(H_\alpha,H_0)$ denote the low-energy and high-energy components of the operator $\mathcal{W}_-(H_\alpha,H_0)$, as defined in  \eqref{eq-low energy wave operator} and \eqref{eq-high energy wave operator}, respectively. The theorem follows by combining the results from Proposition \ref{theorem-low energy-d=1,2}, \ref{theorem-high energy-d=1,2} and \ref{theorem-low energy-d=1,2 unbounded} below.

		\subsection{The low energy contribution}

		We first analyze the properties of $F^\pm(\lambda^2)$ and $G^{\alpha}_{+}(\lambda)$ as $\lambda\to0$. Similar results have been established in \cite{CHZ23}.
		\begin{lemma}\label{lemma-F low energy-d=1,2}
			Let $d=1, 2$, and assume  \eqref{eq-decay condition} holds. For $0<\lambda<\frac{1}{2}$, the following expansions hold: \\
			(i)  For $d=1,$ 
			\begin{equation}\label{eq-F expansion-d=1}
				F^\pm(\lambda^2)=\frac{\pm i}{2}\left(\int_\mathbb{R}\varphi(x)\d x\right)^2\cdot \lambda ^{-1}+b_1+r_1^\pm(\lambda),
			\end{equation}
			where~$b_1=-\frac{1}{2}\big\langle \int_{\mathbb{R}}|x-y|\varphi(y)\d y,\varphi(x)\big\rangle\neq 0$, and the remainder term $r_1^\pm(\lambda)$ satisfies
			\begin{equation}\label{eq: est for r_1}
				\Big|\frac{\d^k}{\d\lambda^k}r_1^\pm(\lambda)\Big|\lesssim \, \lambda^{\frac{1}{2}-k},\quad k=0,1,2.
			\end{equation}
			(ii)  For $d=2,$ 
			\begin{equation}\label{eq-F expansion-d=2}
				F^\pm(\lambda^2)=\left( \frac{\pm i}{4}-\frac{1}{2\pi}\gamma+\frac{1}{2\pi}\log2-\frac{1}{2\pi}\log\lambda\right)\left( \int_{\mathbb{R}^2}\varphi(x)\d x\right)^2+b_1+r_2^\pm(\lambda),
			\end{equation}
			where~$b_1=-\frac{1}{2\pi}\big\langle \int_{\mathbb{R}^2} \log|x-y|\varphi(y)\d y,\varphi(x)\big\rangle \neq 0$, and   the remainder term $r_2^\pm(\lambda)$  satisfies
			\begin{equation}\label{eq: est for r2}
				\Big|\frac{\d^k}{\d\lambda^k}r_2^\pm(\lambda)\Big|\lesssim \, \lambda^{\frac{1}{2}-k},\quad k=0,1,2.
			\end{equation}
		\end{lemma}
		\begin{proof}
			We first consider the case $d=1$. Recall that
			\begin{equation}\label{eq-free-1d-6-22}
				R^\pm_0(\lambda^2;x,y)=\frac{\pm i}{2\lambda}e^{\pm i\lambda|x-y|}.
			\end{equation}
			Expanding $R^\pm_0(\lambda^2;x,y)$ in a Taylor series yields
			\begin{equation}\label{eq: exp for reso d=1}
				R^\pm_0(\lambda^2;x,y)=\pm\frac{i}{2\lambda}- \frac 12 |x-y|+ r_{1}^{\pm}(\lambda, |x-y|),
			\end{equation}
			where the remainder term $r_{1}^{\pm}(\lambda, |x-y|)$ satisfies 
			\begin{equation}\label{eq: est for the int kernel of r_1}
				\Big|\partial_\lambda^{k} r_{1}^{\pm}(\lambda, |x-y|)\Big|\lesssim \lambda^{1-k}|x-y|^{2}, \quad \,\,  k=0,1,2. 
			\end{equation}
			This implies \eqref{eq-F expansion-d=1}.
			Furthermore, combining  \eqref{eq: est for the int kernel of r_1} with the decay assumption $|\varphi(x)| \lesssim \langle x \rangle^{-3-}$  in \eqref{eq-decay condition} yields \eqref{eq: est for r_1}.
			
			For $d=2$, the integral kernel $R_0^\pm(\lambda^2; x,y)$  admits the expansion (see e.g. in \cite[(2.14)]{CHZ23})
			\begin{equation}\label{eq: exp for reso d=2}
				R_0^\pm(\lambda^2; x,y)=\left( \frac{\pm i}{4}-\frac{1}{2\pi}\gamma+\frac{1}{2\pi}\log2-\frac{1}{2\pi}\log\lambda\right)-\frac{1}{2\pi}\log|x-y|+r_2^\pm(\lambda,|x-y|),
			\end{equation}
			where 
			\begin{equation}\label{eq: exp for reso d=2-76}
				\Big|\partial_\lambda^{k} r_{2}^{\pm}(\lambda, |x-y|)\Big|\lesssim \lambda^{1-k}\langle x-y\rangle^{\frac 32}, \quad  k=0,1,2. 
			\end{equation}
			The decay condition 
			$|\varphi(x)|\lesssim \langle x\rangle^{-4 -}$ from \eqref{eq-decay condition} yields \eqref{eq-F expansion-d=2} and \eqref{eq: est for r2}
		\end{proof}
		
		As an immediate consequence, we have
		\begin{lemma}\label{lem: for int donse not vanish}
			Let $d=1, 2$ and let \eqref{eq-decay condition} hold. The following statements hold for $0<\lambda<\lambda_0$ with a small fixed $\lambda_0>0$:\\
			(i) If $\int_{\mathbb{R}^d}{\varphi\,\d x}=0$,  then for $k=0,1,2$, we have
			\begin{equation}\label{eq-G-pm0501}
				\Big|\frac{\d^k}{\d\lambda^k}G_+^{\alpha}(\lambda)\Big|\lesssim \lambda^{-k}.
			\end{equation}
			(ii) If $\int_{\mathbb{R}^d}{\varphi\,\d x}\neq 0$ and $d=1$, we have
			\begin{equation}\label{eq: expansion for G when d=1}
				G^{\alpha}_{+}(\lambda)= (a_1^{\pm})^{-1}\lambda+ B_1^{+}(\lambda),	
			\end{equation}
			where $a_1^{\pm}=\frac{\pm i}{2}\big(\int_{\mathbb{R}}\varphi(x)\d x\big)^2$ and $B_1^{\pm}(\lambda)$ satisfies 
			\begin{equation}\label{eq: est for B1}
				\Big|\frac{\d^k}{\d\lambda^k}B_1^{\pm}(\lambda)\Big|\lesssim \lambda^{2-k}, \quad k=0,1,2.
			\end{equation}
			(iii)  If $\int_{\mathbb{R}^d}{\varphi\,\d x}\neq 0$ and $d=2$,  we have
			\begin{equation}\label{eq: expansion for G when d=2}
				G^{\alpha}_{+}(\lambda)=(a_2)^{-1}(\log\lambda)^{-1}+ B_2^{+}(\lambda),
			\end{equation}
			where $a_2=-\frac{1}{2\pi}\big(\int_{\mathbb{R}^2}\varphi(x)\d x\big)^2$ and $B_2^{\pm}(\lambda)$ satisfies
			\begin{equation}\label{eq: est for B2}
				\Big|\frac{\d^k}{\d\lambda^k}B_2^{\pm}(\lambda)\Big|\le C \, |\log\lambda|^{-2} \lambda^{-k}, \quad k=0,1,2.
			\end{equation}
		\end{lemma}

		We now prove the boundedness of the low-energy operator $\mathcal{W}_-^l(H_\alpha,H_0)$ with the $\lambda_0$ be chosen as in Lemma \ref{lem: for int donse not vanish}.
		
		\begin{proposition}\label{theorem-low energy-d=1,2}
			Let $d = 1, 2$, and let $\lambda_0$ be as in Lemma~\ref{lem: for int donse not vanish}. Suppose \textbf{Condition $H_\alpha$} holds. \\
			(i) If $\int_{\mathbb{R}^d}\varphi(x)\d x=0,$ then the low-energy wave operator $\mathcal{W}_-^l(H_\alpha,H_0)$ is bounded on $L^p(\mathbb{R}^d)$ for all $1\le p\le \infty.$ \\
			(ii) If $\int_{\mathbb{R}^d}\varphi(x)\d x\neq0,$ then  $\mathcal{W}_-^l(H_\alpha,H_0)$ is bounded on $L^p(\mathbb{R}^d)$ for all $1< p<\infty$. Additionally, it is bounded from $L^1(\mathbb{R}^d)$ to the weak $L^1$ space  $L^{1,\infty}(\mathbb{R}^d)$.
		\end{proposition}
		\begin{proof}
			We first prove statement (i).
			
			
			\textbf{Case  $d=1$}: By \eqref{eq-free-1d-6-22}
			the integral kernel of 
			$\mathcal{W}_-^l(H_\alpha,H_0)$ is
			\begin{equation}\label{eq-wave operator kernel-d=1}
				\mathcal{W}_-^l(H_\alpha,H_0)(x,y)=L_+(x,y)-L_-(x,y)
			\end{equation}
			where
			\begin{equation}\label{L-pm-0501}
				L_\pm(x,y)=\frac{i}{4\pi }\int_0^\infty e^{i\lambda(|x|\pm|y|)}\lambda^{-1} G^{\alpha}_{+}(\lambda)\chi(\lambda)E^\pm(\lambda;x,y)\d\lambda, 
			\end{equation}
			with
			\begin{equation}\label{E-pm-503}
				E^\pm(\lambda; x,y)=\int_{\mathbb{R}}e^{i\lambda (|x-x_1|-|x|)}\varphi(x_1)\d x_1\int_{\mathbb{R}}e^{\pm i\lambda(|x_2-y|-|y|)}\varphi(x_2)\d x_2.
			\end{equation}
			Using the condition $\int_{\mathbb{R}}\varphi(x)\d x=0$, we can rewrite $E^\pm$ as 
			\begin{align}\label{E-pm-05-1}
				E^\pm(\lambda; x,y)=&\int_{\mathbb{R}}\big(e^{i\lambda (|x-x_1|-|x|)}-1\big)\varphi(x_1)\d x_1\int_{\mathbb{R}}\big(e^{\pm i\lambda(|x_2-y|-|y|)}-1\big)\varphi(x_2)\d x_2\nonumber\\
				=&(i\lambda)^2\int_{\mathbb{R}}\int_0^{|x-x_1|-|x|}e^{i\lambda\theta_1}\varphi(x_1)\d\theta_1\d x_1\int_{\mathbb{R}}\int_{0}^{|x_2-y|-|y|}e^{\pm i\lambda\theta_2}\varphi(x_2)\d\theta_2\d x_2\nonumber\\
				:=&(i\lambda)^2\widetilde{E}^\pm(\lambda;x,y).
			\end{align}
			Under assumption \eqref{eq-decay condition}, a direct computation yields
			\begin{equation}\label{eq-E-pm-0501}
				\sup_{x,y}\left|\partial_{\lambda}^{k}\widetilde{E}^\pm(\lambda;x,y) \right|\lesssim 1, \quad k=0,1,2.
			\end{equation}
			And thus by \eqref{eq-G-pm0501}, there is
			\begin{equation}\label{eq-E-pm-0501-1}
				\sup_{x,y}\Big|\partial_{\lambda}^{k} \left[\lambda^{-1}G^{\alpha}_{+}(\lambda){E}^\pm(\lambda;x,y)\right] \Big|\lesssim \lambda^{1-l}, \quad \lambda\in (0,\lambda_0), \quad  k=0,1,2.
			\end{equation}
			Then by Lemma \ref{oscillatory estimates} with $b=\frac 12$, one has
			\begin{equation}\label{L-pm-0501-2}
				|L_\pm(x,y)|\lesssim\langle |x|\pm|y|\rangle^{-\frac 32}.
			\end{equation}
			This immediately implies  that $L_\pm$ is admissible.

			\textbf{Case  $d=2$}: Recall that
			\begin{equation*}\label{eq-free kernel-d=2}
				R^\pm_0(\lambda^2;|x-y|)=\frac{\pm i}{4}H_0^\pm(\lambda|x-y|),
			\end{equation*}
			and 
			\begin{equation*}\label{eq-free kernel-d=2-2}
				R^+_0(\lambda^2;|x-y|)-R^-_0(\lambda^2;|x-y|)=\frac{i}{2}J_0(\lambda|x-y|).
			\end{equation*}
			Hence,  the kernel of $\mathcal{W}_-^l(H_\alpha,H_0)$ can be written as
			\begin{equation*}\label{eq-wave operator kernel-d=2 low energy}
				\begin{split}
					\mathcal{W}_-^l&(H_\alpha,H_0)(x,y)\\
					=&\frac{i}{ 8\pi }\int_0^\infty\lambda G^{\alpha}_{+}(\lambda)\chi(\lambda)\Big(\int_{\mathbb{R}^2}H^+_0(\lambda|x-x_1|)\varphi(x_1)\d x_1\cdot\int_{\mathbb{R}^2}J_0(\lambda|x_2-y|)\varphi(x_2)\d x_2\Big)\d\lambda.
				\end{split}
			\end{equation*}
			Since $\int_{\mathbb{R}^2}\varphi(x)\d x=0$, we can rewrite $\mathcal{W}_-^l(H_\alpha,H_0)(x,y)$ in the form
			\begin{equation*}\label{eq-wave operator kernel-d=2 low energy-2}
				\begin{split}
					&\mathcal{W}_-^l(x,y)\\
					=&\frac{i}{ 8\pi }\int_{\mathbb{R}^4}\Big(\int_{\langle x\rangle}^{|x-x_1|}\int_{|y|}^{|x_2-y|}\int_0^\infty\lambda^3 G^{\alpha}_{+}(\lambda)\chi(\lambda)(H^+_0)'(\lambda\theta_1)J'_0(\lambda\theta_2)\d\lambda \d\theta_2\d\theta_1\Big)\cdot\varphi(x_1)\varphi(x_2)\d x_1\d x_2.
				\end{split}
			\end{equation*}
			The integral with respect to $\lambda$ is bounded  by (see  \cite[Lemma 3.2]{EGG18})
			\begin{equation*}\label{eq-integral bound}
				\Big|\int_0^\infty\lambda^3 G^{\alpha}_{+}(\lambda)\chi(\lambda)(H^+_0)'(\lambda\theta_1)J'_0(\lambda\theta_2)\d\lambda\Big|\lesssim k(\theta_1,\theta_2),
			\end{equation*}
			where
			$$k(\theta_1,\theta_2):=\frac{1}{\sqrt{\theta_1\theta_2}\langle\theta_1-\theta_2\rangle^2}+\frac{1}{\theta_1\langle\theta_1+\theta_2\rangle^{2+}}.$$
			By interchanging the order of integration for
			$\theta_2$ and $x_2$ and using the assumption \eqref{eq-decay condition}, we obtain
			\begin{align*}
				|\mathcal{W}_-^l(x,y)|
				&\leq  \int_{\R^2}\int_{\langle  x\rangle  }^{|x-x_1|}\int_0^{\infty} \int_{|x_2| \geq |\theta_2-|y||}  k(\theta_1,\theta_2)\left| \varphi(x_1)\right| \left| \varphi(x_2) \right|  
				\,\d x_2 \d\theta_2 \d\theta_1 \d x_1 \\
				&\les \int_{\R^2} \int_{\langle  x\rangle  }^{|x-x_1|}\int_0^{\infty}    k(\theta_1,\theta_2) \langle\theta_2-|y|\rangle  ^{-N} \left| \varphi(x_1)\right| \d\theta_2 \d\theta_1 \d x_1, 
			\end{align*}
			for some $N=\frac{\delta}{2}>2$.
			Noting  $\theta_1>\min\{1,\,|x-x_1|\}$ in the above integral, we have 
			$$\frac{1}{\theta_1 \langle\theta_1+\theta_2\rangle^{2+}}\les \frac{1+|x-x_1|^{-1}}{\langle \theta_1 \rangle \langle \theta_1+\theta_2\rangle^{2+}} \quad  \text{if $\theta_1 \le 1$}.$$
			Applying the same argument to the $\theta_1$ and $x_1$ integrals,   we derive
			\begin{equation}\label{eq:to tilde k}
				\begin{aligned}
					&\int_{\R^2} \int_{\langle  x\rangle  }^{|x-x_1|}\int_0^{\infty}  k(\theta_1,\theta_2) \langle\theta_2-|y|\rangle  ^{-N} \left| \varphi(x_1)\right| \d\theta_2 \d\theta_1 \d x_1\\
					&\les \int_0^\infty \int_0^{\infty} \tilde{k}(\theta_1,\theta_2) \langle\theta_1-|x|\rangle  ^{-N} \langle\theta_2-|y|\rangle^{-N}\d\theta_2 \d\theta_1,
				\end{aligned}    
			\end{equation}
			where
			$$\tilde{k}(\theta_1,\theta_2)=k(\theta_1,\theta_2):=\frac{1}{\sqrt{\theta_1\theta_2}\langle\theta_1-\theta_2\rangle^2}+\frac{1}{\langle \theta_1\rangle \langle\theta_1+\theta_2\rangle^{2+}}.$$
			In the above inequality, we use the inequality
			\begin{align*}
				\int_{\R^2}\int_{\langle  x\rangle  }^{|x-x_1|} \frac{1}{\theta_1 \langle\theta_1+\theta_2\rangle^{2+}} \left| \varphi(x_1)\right| \d\theta_1 \d x_1 
				&\les  \int_{\R^2}\int_{\langle  x\rangle  }^{|x-x_1|} \frac{\langle\theta_1-|x|\rangle^{-N}}{\theta_1 \langle\theta_1+\theta_2\rangle^{2+}}  \left| |x_1|^N \varphi(x_1)\right| \d\theta_1 \d x_1  \\
				&\les \int_{\R^2} \int_0^\infty \frac{\langle\theta_1-|x|\rangle^{-N} (1+|x-x_1|^{-1})}{\langle \theta_1 \rangle \langle \theta_1+\theta_2\rangle^{2+}}  \left| |x_1|^N \varphi(x_1)\right| \d\theta_1 \d x_1 \\
				&\les \int_0^\infty \frac{1}{\langle \theta_1 \rangle \langle \theta_1+\theta_2\rangle^{2+}}\langle\theta_1-|x|\rangle^{-N}\d \theta_1.
			\end{align*}
			Note that 
			$$
			\int_{\mathbb{R}^2}{\langle\theta_1-|x|\rangle  ^{-N} \, \d x} \lesssim \langle \theta_1 \rangle
			$$
			holds for $N>2$.
			Applying Lemma \ref{lemma-GV}, we have
			\begin{align*}
				\sup_{y\in\mathbb{R}^2}\int |\mathcal{W}_-^l(x,y)|\d x	\le \sup_{y\in\mathbb{R}^2}\int_0^\infty \int_0^{\infty} \tilde{k}(\theta_1,\theta_2) \langle \theta_1 \rangle  \langle\theta_2-|y|\rangle^{-N} \d\theta_2 \d\theta_1   \lesssim 1.
			\end{align*}
			Similarly, 
			\begin{align*}
				\sup_{x\in\mathbb{R}^2}\int |\mathcal{W}_-^l(x,y)|\d y\le 	\sup_{x\in\mathbb{R}^2}\int_0^\infty \int_0^{\infty} \tilde{k}(\theta_1,\theta_2) \langle \theta_2 \rangle\langle\theta_1-|x|\rangle  ^{-N} \d\theta_2 \d\theta_1   \lesssim 1.
			\end{align*}
			Hence, the kernel $\mathcal{W}_-^l(x,y)$ is admissible, which completes  the proof for the case $d=2$.
			
			Next, we prove statement (ii).
			
			\textbf{Case   $d=1$:} 
			Recall that the integral kernel of $\mathcal{W}_-^l(H_\alpha,H_0)$ is given by \eqref{eq-wave operator kernel-d=1}-\eqref{E-pm-503}. Using  expansion \eqref{eq: expansion for G when d=1}, we write
			\begin{align}\label{W-pm-05-3}
				\mathcal{W}_-^l(H_\alpha,H_0)(x,y)=&\frac{i}{ 4\pi a_1^{+}}\left(\int_0^\infty e^{i\lambda(|x|+|y|)}\chi(\lambda) E^+(\lambda;x,y)\d\lambda-\int_0^\infty e^{i\lambda(|x|-|y|)}\chi(\lambda) E^-(\lambda;x,y)\d\lambda\right)\nonumber\\
				+&\frac{i}{4\pi }\left(\int_0^\infty e^{i\lambda(|x|+|y|)}\frac{B_1^{+}(\lambda)}{\lambda} E^+(\lambda;x,y)\d\lambda-\int_0^\infty e^{i\lambda(|x|-|y|)}\frac{B_1^{+}(\lambda)}{\lambda}  E^-(\lambda;x,y)\d\lambda\right)\nonumber\\
				:=&\mathcal{W}_-^{l,1}(x,y)+\mathcal{W}_-^{l,2}(x,y)
			\end{align}
			where 
			$$a_1^{+}=\frac{i}{2}\left(\int\varphi(x)\d x\right)^2\ne 0$$
			and $E^+(\lambda; x,y)$ is given by \eqref{E-pm-503}.
			Under the decay assumption \eqref{eq-decay condition} and the fact $||x-x_1|-|x||\le |x_1|$, we obtain the uniform bound
			\begin{equation}\label{eq: est for E}
				\sup_{x,y}\left| \partial_{\lambda}^{k}E^\pm(\lambda;x,y)\right|\lesssim 1, \quad k=0,1,2.	
			\end{equation}
			
			We first estimate the term $\mathcal{W}_-^{l,1}(x,y)$. When $\big||x|-|y|\big|\le 1$, we  immediately obtain
			\begin{equation}\label{eq-Wl1-0504-2}
				\left|\mathcal{W}_-^{l,1}(x,y) \right|\lesssim 1 \lesssim   \langle |x|\pm |y|\rangle^{-2},
			\end{equation}
			which is an admissible integral kernel. In the region where $\big||x|-|y|\big|>1$, integration by parts gives
			\begin{align}\label{eq-Wl1-0504-1}
				\mathcal{W}_-^{l,1}(x,y)=&\frac{1}{2\pi i}\Big(\frac{1}{|x|+|y|}-\frac{1}{|x|-|y|}\Big) 
				-\frac{1}{ 4a_1^{+}\pi }\frac{1}{|x|+|y|}\int_0^\infty e^{i\lambda(|x|+|y|)} \left[ \chi(\lambda) E^+(\lambda;x,y)\right]'\d\lambda \nonumber\\
				+&\frac{1}{ 4a_1^{+}\pi }\frac{1}{|x|-|y|}\int_0^\infty e^{i\lambda(|x|-|y|)}\left[ \chi(\lambda) E^-(\lambda;x,y)\right]'\d\lambda.\nonumber\\
				:=&\uppercase\expandafter{\romannumeral1}(x,y)+\uppercase\expandafter{\romannumeral2}(x,y)+\uppercase\expandafter{\romannumeral3}(x,y).
			\end{align}
			Applying integration by parts again to $\uppercase\expandafter{\romannumeral2}$ and $\uppercase\expandafter{\romannumeral3}$, and  invoking  \eqref{eq: est for E}, we have 
			\begin{equation}\label{eq-I-123-0504}
				|\uppercase\expandafter{\romannumeral2}(x,y)|+|\uppercase\expandafter{\romannumeral3}(x,y)|\lesssim \langle|x|-|y|\rangle^{-2},
			\end{equation}
			which is admissible. The singular term $\uppercase\expandafter{\romannumeral1}$ corresponds to the kernel of the truncated Hilbert transform, which is known to be bounded on $L^p(\mathbb{R})$ for all $1< p<\infty$ and from $L^1(\R)$ to $L^{1,\infty}(\R)$ (see e.g. in \cite{gra}).

			Next, we deal with the term $\mathcal{W}_-^{l,2}(x,y)$. From \eqref{eq: est for B1} and the uniform bound \eqref{eq: est for E}, we have
			$$\sup_{x,y} \left| \partial_{\lambda}^{k}\left(\lambda^{-1} E^\pm(\lambda;x,y)B_1^{+}(\lambda)\right) \right|\lesssim \lambda^{1-k},\quad k=0,1,2.$$
			This reduces to the scenario that $\int_{\mathbb{R}}{\varphi(x)\d x=0}$. Employing the same method as in the proof of \eqref{L-pm-0501-2} in statement  (i), we obtain the bound
			\begin{equation}\label{eq-W-l2-0504}
				|\mathcal{W}_-^{l,2}(x,y)|\lesssim \langle |x|\pm |y|\rangle^{-\frac 32},   
			\end{equation}
			which is admissible. This finishes the case for $d=1$.
			

			\textbf{Case   $d=2$:}   Rather than explicitly writing the kernel of 
			$\mathcal{W}_-^l(H_\alpha,H_0)$, we note that it can be expressed as the composition
			\begin{equation*}
				\mathcal{W}_-^l(H_\alpha,H_0)=T_\varphi\circ \left(G_+^{\alpha}(\sqrt{-\Delta})\chi(\sqrt{-\Delta})\right),
			\end{equation*}
			where $T_\varphi$ is defined by
			\begin{equation}\label{eq-T-phi-0504}
				T_\varphi=\frac{1}{\pi i}\int_0^\infty \lambda R_0^+(\lambda^2)\varphi\langle (R_0^+(\lambda^2)-R_0^-(\lambda^2))\cdot,\varphi\rangle \d\lambda,
			\end{equation}
			and $G_+^{\alpha}(\sqrt{-\Delta})\chi(\sqrt{-\Delta})$ is defined via the Fourier multiplier $G_+^{\alpha}(|\xi|)\chi(|\xi|)$.
			The results follow directly  by combining the following Lemma \ref{lem-T-phi-0504}  and \ref{lem-Fou-multiplier}, where the boundedness of both  $T_\varphi$  and $G_+^{\alpha}(\sqrt{-\Delta})\chi(\sqrt{-\Delta})$ is established.
		\end{proof}
		
		\begin{lemma}\label{lem-T-phi-0504} 
			The operator $T_\varphi$, defined in \eqref{eq-T-phi-0504}, is bounded on $L^p(\mathbb{R}^2)$ for $1<p<\infty$ and maps $L^1(\mathbb{R}^2)$ to the weak $L^1$ space  $L^{1,\infty}(\mathbb{R}^2)$.
		\end{lemma}
		\begin{proof}
			We first note that for any Schwartz function $f$, 
			\begin{equation}\label{eq-T-phi-5-5}
				\widehat{T_\varphi f}(\xi)=\lim_{\varepsilon\downarrow0} \frac{1}{\pi i}\int_0^\infty \frac{\widehat{\varphi}(\xi)}{|\xi|^2-\lambda^2-i\varepsilon} \big\langle (R_0^+(\lambda^2)-R_0^-(\lambda^2)) f,\, \varphi \big\rangle \lambda \d\lambda,  
			\end{equation}
			and observe that the spectral measure of the Laplacian in $\R^2$ can be written as  
			$$ 
			\frac{\lambda}{\pi i} \big\langle R_0^+(\lambda^2)-R_0^-(\lambda^2)) f, \, \varphi \big\rangle  =  \int_{|\eta|=\lambda} \widehat{f}(\eta)\bar{\widehat{\varphi}}(\eta)\d\sigma_{\lambda},
			$$
			where $\d\sigma_{\lambda}$ denotes the surface measure of the sphere $\{\eta\in\R^2,\, |\eta|=\lambda\}$.
			Inserting this into \eqref{eq-T-phi-5-5} and taking the inverse Fourier transform, then the integral kernel of $T_\varphi$ has the form
			\begin{equation}\label{rep-for-0-s-1}
				T_\varphi(x, y)=\frac{1}{(2\pi)^4}
				\lim_{\varepsilon\downarrow0}\int\int\int \int\frac{1}{|\xi|^2-|\eta|^2-i\varepsilon}e^{i\xi\cdot(x-w)}\varphi(w)\,  \varphi(z)e^{i\eta\cdot(z-y)} \d\xi \d w  \d\eta \d z. 
			\end{equation}
			Since for any $\varepsilon>0$,
			$$\frac{i}{|\xi|^2-|\eta|^2-i\varepsilon}=\int_{0}^{\infty}e^{-it\left( |\xi|^2-|\eta|^2-i\varepsilon\right)}\d t.$$
			Combining this with \eqref{rep-for-0-s-1} yields 
			\begin{equation}\label{rep-for-0-s-2}
				T_\varphi(x, y)=\frac{1}{(4\pi)^2}
				\lim_{\varepsilon\downarrow0}\int_{\mathbb{R}^2}\int_{\mathbb{R}^2} \int_0^{\infty} e^{i\frac{|x-w|^2-|z-y|^2}{4t}-\varepsilon t} \varphi(w) \varphi(z) t^{-2}\d t \d w\d z. 
			\end{equation}
			Here we have used the identity
			$$\left( e^{-it|\xi|^{2}}\right)^{\vee}(x)=(4\pi i t)^{-1}\, e^{i\frac{|x|^2}{4t}}. $$
			Let $H(x)$ denote the Heaviside function, i.e., $H(x)=1,$ when $x\ge 0$ and $H(x)=0$, when $x<0$. It is well known that
			$$
			\hat{H}(\xi)=-iPV\frac{1}{\xi}+\pi \delta
			$$
			holds in the sense of distribution,
			where $PV\frac{1}{\xi}$ denotes the homogeneous distribution of degree $-1$.  Thus, applying the change of variables $s=\frac1t$ in \eqref{rep-for-0-s-2} leads to 
			$$T_\varphi(x, y)=\frac{1}{(4\pi)^2}
			PV\int_{\mathbb{R}^2}\int_{\mathbb{R}^2} \left( \frac{1}{|x-w|^2-|z-y|^2}+\pi \delta(|x-w|^{2}-|z-y|^2)\right) \varphi(w) \varphi(z)\d w\d z.$$
			Since the spaces $L^1$, $L^{1,\infty}$ are translation invariant,
			it suffices to consider the integral kernel
			$$PV\frac{1}{|x|^2-|y|^2}+\pi\delta(|x|^2-|y|^2)$$ 
			For each fixed  $x$, $\delta(|x|^2-|y|^2)$ corresponds to the surface measure of the sphere $|y|=|x|$, scaled by $\frac{1}{2|x|}$.
			Specifically,
			$$\int\delta(|x|^2-|y|^2)f(y)\d y=\frac{1}{2|x|}\int_{|y|=|x|}f(y)\d y.$$
			Consequently, the averaging operator with  kernel $\delta(|x|^2-|y|^2)$ is bounded on $L^p$ for $1\le p \le \infty$.
			Let $\mathbf{T}$ be the operator associated with the integral kernel $PV\frac{1}{|x|^2-|y|^2}$, then we have 
			\begin{align*}
				(\mathbf{T}f)(x)=PV\int \frac{1}{|x|^2-|y|^2} f(y)\d y&= \frac{1}{2|x|}	PV\int \frac{1}{|x|-|y|} f(y)\d y+\frac{1}{2|x|}\int \frac{1}{|x|+|y|} f(y)\d y\\
				&:=  (\mathbf{T_1}f)(x)+(\mathbf{T_2}f)(x)
			\end{align*}
			The second term satisfies
			\begin{equation}
				\left|(\mathbf{T_2}f)(x)\right|	=\frac{1}{2|x|}\left|\int \frac{1}{|x|+|y|} f(y)\d y\right|\le \frac{1}{2|x|^2}\left\|f \right\|_{L^1} .	
			\end{equation}
			This implies
			\begin{equation}\label{eq-weak-1-6-23}
				\|\mathbf{T_2}f\|_{L^{1,\infty}}	\lesssim \left\|f\right\|_{L^1}.
			\end{equation}
			On the other hand, observe that $(\mathbf{T_1}f)(x):=g(r)$ (where $r=|x|$) is radial, Its norm admits  the equivalence
			$$
			\left\| \mathbf{T_1}f\right\|_{L^{1,\infty}(\mathbb{R}^2)}=2\pi \left\| g(r)\right\|_{L^{1,\infty}(\mathbb{R}^+, r\d r)}.
			$$
			Expressing $\mathbf{T_1}f$ in polar coordinates gives
			$$
			g(r)=\frac{1}{2r}\int_0^{\infty} \frac{1}{r-\rho} \left(\int_{\mathbf{S}^1} f(\rho\theta) \d\theta\right) \rho\d\rho,
			$$
			Applying  the $L^1-L^{1,\infty}$ boundedness of the Hilbert transform, we obtain
			\begin{equation*}
				\left\| g(r)\right\|_{L^{1,\infty}(\mathbb{R}^+, r\d r)}	=\pi \left\|	PV\int \frac{1}{r-\rho} \int_{\mathbf{S}^1} f(\rho\theta) \d\theta \rho\d\rho \right\|_{L^{1,\infty}(\mathbb{R}^+, dr)} 
				\lesssim \left\|f\right\|_{L^1},
			\end{equation*} 
			Combining these estimates for $\mathbf{T_1}f$, we derive that
			\begin{equation}\label{eq-weak-1-6-23-2}
				\|\mathbf{T_1}f\|_{L^{1,\infty}}	\lesssim \left\|f\right\|_{L^1}.
			\end{equation}
			\eqref{eq-weak-1-6-23} and \eqref{eq-weak-1-6-23-2} yield weak type $(1, 1)$ boundedness for $\mathbf{T}$.  
			Similarly, we obtain that  $\mathbf{T}$  is bounded on $L^p(\R^2)$ for $1<p<\infty$. Therefore We  have established the desired result of $\mathcal{W}_-^l(H_\alpha,H_0)$.
		\end{proof}

		We now examine the boundedness of $G_+^{\alpha}(\sqrt{-\Delta})\chi(\sqrt{-\Delta})$ through a more general framework. Consider the Fourier multiplier operator defined by
		\begin{equation}\label{eq-def-fou-mul}
			\mathbf{M}_af=\mathcal{F}^{-1}\big(\psi(\cdot)\chi(|\cdot|)\hat{f}(\cdot)\big), \quad f\in \mathscr{S}(\mathbb{R}^d).
		\end{equation}
		where the symbol $\psi(\xi) \in C^{d+1}$  satisfies 
		\begin{equation}\label{eq-psi-0504}
			\left|\partial^{\gamma} \psi(\xi) \right|\lesssim 
			\begin{cases}
				|\log|\xi||^{-a},& \quad  |\gamma|=0,\\
				|\log|\xi||^{-a-1}|\xi|^{-|\gamma|},& \quad  0<|\gamma|\le d+1,\\
			\end{cases}
		\end{equation}
		for a constant $a \in \mathbb{R}$. 
		Note that $G_+^{\alpha}(\lambda)$ corresponds to the special case where $\psi(\lambda)=\frac{\alpha}{1+\alpha F^{+}(\lambda^2)}$. Moreover, from \eqref{eq: expansion for G when d=2} and \eqref{eq: est for B2}, this particular $\psi$ satisfies condition \eqref{eq-psi-0504} with $a=1$. 
		We note that  Lemma \ref{lem: Fourier Multiplier by Bernstein}  cannot be applied directly to conclude that    $\psi \in M_1(\mathbb{R}^d)$, because $\psi$ fails to satisfy \eqref{eq:condition for FM} (note that \eqref{eq-psi-0504} only provides additional logarithmic decay). Nevertheless, we can still derive the desired estimate through a direct computation. More precisely,  we have
		
		\begin{lemma}\label{lem-Fou-multiplier} 
			Let $\mathbf{M}_a$ be defined in \eqref{eq-def-fou-mul} and \eqref{eq-psi-0504}	hold with  $a> \frac1d$. Then $\mathbf{M}_a$ is bounded on $L^p(\mathbb{R}^d)$ for all $1\le p\le \infty$.
		\end{lemma}
		\begin{proof}
			The boundedness follows if we prove that
			\begin{equation}\label{eq-fou-L1}
				\mathcal{F}^{-1}\big(\psi(\cdot)\chi(|\cdot|)\big)(x)\in L^1(\mathbb{R}^d),
			\end{equation}
			which we establish via the stationary phase method.
			
			For $|x|\geq 100$, set $r_0=|x|^{-1}\big|\log|x|\big|^{-\frac{1}{d}}$ and decompose
			\[
			\psi(\xi)\chi(|\xi|) = \psi(\xi)\chi(|\xi|)\chi(|\xi|/r_0) + \psi(\xi)\chi(|\xi|)(1-\chi(|\xi|/r_0)).
			\]
			Applying \eqref{eq-psi-0504} with $ |\gamma|=0$ and using a direct size estimate, we obtain
			\[
			\left|\big(\psi(\xi)\chi(|\xi|)\chi(|\xi|/r_0)\big)^{\vee}(x)\right| \leq |\log r_0|^{-a}r_0^{d}.
			\]
			For the second term,  integration by parts gives
			\begin{align*}
				&\big[\psi(\xi)\chi(|\xi|)(1-\chi(|\xi|/r_0))\big]^{\vee}(x) \\
				&= \frac{i}{(2\pi)^d}\int e^{ix\cdot\xi}\left[\frac{x}{|x|^2}\cdot\nabla(\psi\chi)(1-\chi(r_0^{-1}|\xi|)) + \frac{x}{|x|^2}\cdot(\psi\chi)\nabla\chi(r_0^{-1}|\xi|)\right]d\xi \\
				&= I + II.
			\end{align*}
			Since $\big|\nabla\chi(|\xi|/r_0)\big|\lesssim r_0^{-1}$ on the region $\{|\xi|\sim r_0\}$, we obtain
			\[
			|II| \lesssim \big|\log r_0\big|^{-a}|x|^{-1}r_0^{d-1}.
			\]
			Moreover, since
			\[
			\sup_{\xi}\big|\partial^\gamma\big[\frac{x}{|x|^2}\cdot\nabla(\psi\chi)(1-\chi(r_0^{-1}|\xi|))\big]\big| \lesssim \big|\log|\xi|\big|^{-a-1}|\xi|^{-1-|\gamma|},
			\]
			performing $d$ additional integrations by parts yields
			\[
			|I| \lesssim \big|\log r_0\big|^{-a-1}|x|^{-(d+1)}r_0^{-1}.
			\]
			Combining these with $\big|\log r_0\big|\geq \log|x|$ for $|x|\geq 100$ gives
			\begin{align*}
				\big|\big(\psi(\xi)\chi(|\xi|)\big)^{\vee}(x)\big| &\lesssim \big|\log r_0\big|^{-a-1}|x|^{-(d+1)}r_0^{-1} + \big|\log r_0\big|^{-a}|x|^{-1}r_0^{d-1} + \big|\log r_0\big|^{-a}r_0^{d} \\
				&\lesssim \big|\log|x|\big|^{-a-1+\frac{1}{d}}|x|^{-d}.
			\end{align*}
			Since $a>\frac{1}{d}$, this implies \eqref{eq-fou-L1}. The proof is complete.
		\end{proof}

		\subsection{The high energy contribution}
		\begin{proposition}\label{theorem-high energy-d=1,2}
			Let $d=1, 2$  and assume that $\varphi$ satisfies \textbf{Condition} $H_\alpha.$ Then the high energy operator $\mathcal{W}_-^h(H_\alpha,H_0)$ (defined in \eqref{eq-high energy wave operator})  is bounded on $L^p(\mathbb{R}^d)$ for all $1\le p\le\infty.$ 
		\end{proposition}
		\begin{proof}
			We first consider  $d=1$. From \eqref{eq-high energy wave operator-2} and \eqref{eq2.1}, the kernel of $\mathcal{W}_-^h(H_\alpha,H_0)$ takes the form
			\begin{equation*}
				\begin{split}
					&\mathcal{W}_-^h(H_\alpha,H_0)(x,y)\\
					=&\frac{i}{4\pi }\int_0^\infty \lambda^{-1}G_+^{\alpha}(\lambda)\widetilde{\chi}(\lambda)\Big(\int_{\mathbb{R}}e^{i\lambda|x-x_1|}\varphi(x_1)\d x_1\cdot\int_{\mathbb{R}}\big(e^{i\lambda|x_2-y|}-e^{-i\lambda|x_2-y|}\big)\varphi(x_2)\d x_2\Big)\d\lambda.\\
				\end{split}
			\end{equation*}
			Now we analyze the oscillatory integral
			$$I(\rho_{+})=\int_0^\infty\lambda^{-1}G_+^{\alpha}(\lambda)\widetilde{\chi}(\lambda)e^{i\lambda\rho_+} \d\lambda,$$
			where we denote 
			$$\rho_{\pm}=|x-x_1|\pm |x_2-y|.$$

			Firstly, we consider the case when $\lambda\rho_{+}\le 1$, we divide $I(\rho_{+})$ into two terms as follows:
			\begin{equation*}
				I_1(\rho_{+})=\int_0^\infty \lambda^{-1}G_+^{\alpha}(\lambda)\widetilde{\chi}(\lambda)e^{i\lambda\rho_{+}}\eta\big(\lambda\rho_{+}\big)\d\lambda, 
			\end{equation*}
			and
			\begin{equation*}
				I_2(\rho_{+})=\int_0^\infty \lambda^{-1}G_+^{\alpha}(\lambda)\widetilde{\chi}(\lambda)e^{i\lambda\rho_{+}}\widetilde{\eta}\big(\lambda\rho_{+}\big)\d\lambda, 
			\end{equation*}
			where $\eta$ is defined in \eqref{eq-cut-off function eta}, and $\widetilde{\eta}=1-\eta.$
			
			For the first term $I_1(\rho_{+})$, using \eqref{lem:est for high G}, a direct computation shows that 
			$$\left| I_1(\rho_{+})\right| \lesssim_{\lambda_0} \int_{\lambda_0/2}^{\rho_{+}^{-1}}\lambda^{-\frac 12}\d\lambda\lesssim_{\lambda_0}(\rho_{+})^{-\frac 12}.  $$
			
			For the second term $I_2(\rho_{+})$, applying the integration by parts once, we have 
			$$\left| I_2(\rho_{+})\right| \lesssim_{\lambda_0}  \rho_{+}^{-1}\int_{\rho_{+}^{-1}}^{\infty}\lambda^{-\frac 32}\d\lambda\lesssim \rho_{+})^{-\frac 12}.$$
			Thus, when $\rho_{+}\le 1$, we can obtain that
			\[\left|I(\rho_{+})\right|\lesssim(\rho_{+})^{-\frac 12}.\]

			Now we study the case when $\lambda\rho_{+}> 1$, based on \eqref{lem:est for high G}, integrating by parts two times on $I(\rho_{+})$ yields 
			\begin{equation*}
				I(\rho_{+})= (\rho_{+}) ^{-2}\int_{\lambda_0/2}^\infty \lambda^{b-1}\d\lambda\lesssim_{\lambda_0} (\rho_{+}) ^{-2}.
			\end{equation*}
			Similar results hold for the integral
			$$\int_0^\infty \lambda^{-1} G_+^\alpha(\lambda) \widetilde{\chi}(\lambda) e^{i\lambda\rho_{-}} \, \d\lambda.$$
			Thus, we have 
			\begin{align*}
				\left|\mathcal{W}_-^h(H_\alpha,H_0)(x,y)\right|&\lesssim
				\int_{\R}\int_{\R} \Big( \min\left\{(|x-x_1|+|x_2-y|)^{-\frac 12},\, (|x-x_1|+|x_2-y|)^{-2}\right\} \\
				&+\min\left\{\big||x-x_1|-|x_2-y|\big|^{-\frac 12},\, \big||x-x_1|-|x_2-y|\big|^{-2}\right\}\Big)  |\varphi(x_1)||\varphi(x_2)|\d x_1 \d x_2.
			\end{align*}
			Note that 
			\begin{align*}
				&\sup_{x_1, x_2, y\in \R}\int_{\R}  \min\left\{\big||x-x_1|\pm |x_2-y|\big|^{-\frac 12},\, \big||x-x_1|\pm |x_2-y|\big|^{-2}\right\} \d x \\
				=&  \sup_{x_1, x_2, y\in \R}\int_{\R}  \min\left\{\big||x|\pm |x_2-y|\big|^{-\frac 12},\, \big||x|\pm |x_2-y|\big|^{-2}\right\} \d x\\
				\lesssim &  \sup_{x_1, x_2, y\in \R}\int_{\R}  \min\left\{|x|^{-\frac 12},\, |x|^{-2}\right\} \d x \\
				\lesssim &1.
			\end{align*}
			This implies that
			$$
			\int_{\R}|\mathcal{W}_-^h(H_\alpha,H_0)(x,y)|\d x\lesssim\int_{\R}\int_{\R}|\varphi(x_1)||\varphi(x_2)|\d x_1 \d x_2\lesssim 1. 
			$$
			Similarly, we also have
			$$
			\int_{\R}|\mathcal{W}_-^h(H_\alpha,H_0)(x,y)|\d y\lesssim\int_{\R}\int_{\R}|\varphi(x_1)||\varphi(x_2)|\d x_1 \d x_2\lesssim 1. 
			$$
			Thus the integral kernel  $\mathcal{W}_-^h(H_\alpha,H_0)(x,y)$ is admissible.

			Next, we prove the result  for $d=2$. By \eqref{eq-free kernel-3}, one has 
			\begin{equation}
				\begin{split}
					R_0^\pm(\lambda^2;x,y)= e^{\pm i\lambda|x-y|}w^{\pm}(\lambda|x-y|),
				\end{split}
			\end{equation}
			where  $w^{\pm}(z)$ satisfies
			\begin{align}\label{eq: est for omega}
				\Big|\frac{\d^k}{\d z^k}w^{\pm}(z)\Big|\lesssim_k |z|^{-\frac 12-k},\quad k\in\mathbb{N}_0.
			\end{align}
			Then we can express the kernel of $\mathcal{W}_-^h(H_\alpha,H_0)$ as 
			\begin{equation}
				\begin{split}
					\mathcal{W}_-^h(H_\alpha,H_0)(x,y)
					=&\frac{i}{\pi }\int_0^\infty\lambda G_+^{\alpha}(\lambda)\widetilde{\chi}(\lambda)\Big(\int_{\mathbb{R}}e^{i\lambda|x-x_1|}\varphi(x_1)w^+(\lambda|x-x_1|)\d x_1\\
					& \cdot\int_{\mathbb{R}}\big(e^{i\lambda|x_2-y|}w^{+}(\lambda|x_2-y|)-e^{-i\lambda|x_2-y|}w^{-}(\lambda|x_2-y|)\big)\varphi(x_2)\d x_2\Big)\d\lambda.\\
				\end{split}
			\end{equation}
			We remark that $G_+^{\alpha}(\lambda)$  satisfies the bounds  provided in Lemma  \ref{lem:est for high G} for $0\le k\le 3$.
			Adapting the method used for dimensions $d\ge 3$, we perform integration by parts once using the operators $L_{x_j}$ and $L_{x_j}^*$ ($j=1, 2$) as defined in \eqref{eq-differential operators} and \eqref{eq-differential operators-dual}.
			This  allows us to rewrite  $\mathcal{W}_-^h(H_\alpha,H_0)(x,y)$  in the form  
			\begin{align*}
				&\int_0^\infty\lambda^{-1} G_+^{\alpha}(\lambda)\widetilde{\chi}(\lambda)\Big(\int_{\mathbb{R}}e^{i\lambda|x-x_1|}\varphi(x_1)U^{+}(\lambda,x,x_1)\d x_1\\
				&\ \ \ \ \ \ \ \ \ \ \ \ \ \ \ \ \ \ \ \ \ \ \ \ \ \ \ \ \ \ \ \ \ \ \ \ \ \ \ \ \ \ \ \ \  \cdot\int_{\mathbb{R}}\big(e^{i\lambda|x_2-y|}U^{+}(\lambda,x_2,y)-e^{-i\lambda|x_2-y|}U^{-}(\lambda,x_2,y)\big)\d x_2\Big)\d\lambda.
			\end{align*}
			By \eqref{eq: est for omega} and the assumption on $\varphi$,  the functions $U^{+}(\lambda,x,x_1)$ and $U^{\pm}(\lambda,x_2,y)$ satisfy 
			\begin{equation}
				\left|\partial_{\lambda}^{k}U^{+}(\lambda,x,x_1)\right|\lesssim \lambda^{-\frac 12-k}|x-x_1|^{-\frac 12}\langle x_1\rangle^{-3},  
			\end{equation}
			and
			\begin{equation}
				\left|\partial_{\lambda}^{k}U^{\pm}(\lambda,x_2,y)\right|\lesssim \lambda^{-\frac 12-k}|x_2-y|^{-\frac 12}\langle x_2\rangle^{-3}.  
			\end{equation}
			Following the same approach as in the  $d=1$ case,  we derive bounds for the kernel $\mathcal{W}_-^h(H_\alpha,H_0)(x,y)$ by
			\begin{align*}
				\Big|\mathcal{W}_-^h&(H_\alpha,H_0)(x,y)\Big|\lesssim
				\int_{\R^2}\int_{\R^2} \Big( \min\{(|x-x_1|+|x_2-y|)^{-\frac 12},\, (|x-x_1|+|x_2-y|)^{-2}\} \\
				&+\min\{\big||x-x_1|-|x_2-y|\big|^{-\frac 12},\, \big||x-x_1|-|x_2-y|\big|^{-2}\}\Big) |x-x_1|^{-\frac 12}|x_2-y|^{-\frac 12}|\varphi(x_1)||\varphi(x_2)| \d x_1 \d x_2.
			\end{align*}
			Note that 
			\begin{align*}
				&\sup_{x_1\in \R^2}\int_{\R^2}  \min\left\{\big||x-x_1|\pm |x_2-y|\big|^{-\frac 12},\, \big||x-x_1|\pm |x_2-y|\big|^{-2}\right\} |x-x_1|^{-\frac 12} \d x \\ 
				&= 2\pi \sup_{x_1\in \R^2}\int_{0}^\infty  \min\left\{\big| \rho \pm |x_2-y|\big|^{-\frac 12},\, \big|\rho \pm |x_2-y|\big|^{-2}\right\}\,\rho ^{\frac 12}\d\rho \\
				&\lesssim \langle x_2-y\rangle^{\frac 12}. 
			\end{align*}
			Thus, 
			$$\sup_{y}\int_{\R^2} \Big|\mathcal{W}_-^h(H_\alpha,H_0)(x,y)\Big|\d x \lesssim \sup_{y}\int_{\R^2} \int_{\R^2}\int_{\R^2} |\varphi(x_1)||\varphi(x_2)|\langle x_2-y\rangle^{\frac 12} |x_2-y|^{-\frac 12}\d x_1\d x_2\lesssim 1. $$
			Similarly, 
			$$\sup_{x}\int_{\R^2} \Big|\mathcal{W}_-^h(H_\alpha,H_0)(x,y)\Big| \d y \lesssim 1. $$
			Therefore $\mathcal{W}_-^h(H_\alpha,H_0)(x,y)$ is admissible, and the proof is complete.
		\end{proof}
		
		\subsection{The unboundedness at $p=1$}
		In this subsection, we prove that the wave operator $\mathcal{W}_-(H_\alpha,H_0)$ fails to be bounded at the endpoint case $p=1$ for  dimensions  $d=1, 2$, provided that $\int_{\R^d} \varphi(x)\d x\neq 0$.  
		\begin{proposition}\label{theorem-low energy-d=1,2 unbounded}
			Let $d=1$ or $2$ and assume that $\int_{\mathbb{R}^d}\varphi(x)\d x\neq0$. Then: 
			
			(i)  The operator $\mathcal{W}_-(H_\alpha,H_0)$ is unbounded on $L^1(\mathbb{R})$ and $L^\infty(\R)$ for $d=1$;  
			
			(ii)  $\mathcal{W}_-(H_\alpha,H_0)$ is unbounded on $L^1(\mathbb{R}^2)$ for $d=2$.
		\end{proposition}
		\begin{proof}
			To prove the unboundedness of the wave operator $W_-(H_\alpha,H_0)$ in dimensions $d=1$ and $d=2$, we apply Proposition \ref{theorem-high energy-d=1,2}, which  reduces the problem to analyzing its low energy component $\mathcal{W}_{-}^l(H_\alpha,H_0)$.
			
			For the case $d=1$, we observe that from expressions \eqref{W-pm-05-3}, \eqref{eq-Wl1-0504-1}, and estimates \eqref{eq-Wl1-0504-2},  \eqref{eq-I-123-0504},  \eqref{eq-W-l2-0504}, the operators associated with kernels $\uppercase\expandafter{\romannumeral2}(x,y)$, $\uppercase\expandafter{\romannumeral3}(x,y)$ and $\mathcal{W}_-^{l,2}(x,y)$ are bounded on $L^p$ for all $1\le p\le \infty$.
			Thus, we need only examine the  unboundedness of the operator associated with the kernel $\uppercase\expandafter{\romannumeral1}(x,y)$ in \eqref{eq-Wl1-0504-1}.
			Let $\mathbb{H}^{*}_{1}$ be the operator associated with the kernel 
			$$\frac{1}{2\pi i}\Big(\frac{1}{|x|+|y|}-\frac{1}{|x|-|y|}\Big)
			\mathbf{1}_{\{||x|-|y||>1\}},
			$$ 
			where $\mathbf{1}_{\{||x|-|y||>1\}}$ denotes the characteristic function of $\{(x,y)\in\R^2: ||x|-|y||>1\}$. Consider the test function $f(y)=\mathbf{1}_{\{|y|>2\}}$, then for each fixed $x$ with $|x|<1$, we have  
			$$
			(\mathbb{H}^{*}_{1}f)(x)=\frac{1}{2\pi i}\int_{\R}\Big(\frac{1}{|x|+|y|}-\frac{1}{|x|-|y|}\Big)
			\mathbf{1}_{\{||x|-|y||\}>1}f(y) \d y=\frac{1}{\pi i}\int_{|y|>2}\frac{|y|}{|y|^2-|x|^2}\d y,
			$$
			which diverges to $\infty$. This shows that $\mathbb{H}^{*}_{1}$ is  unbounded on $L^\infty(\mathbb{R})$.  By duality, it is  unbounded on $L^1(\mathbb{R})$ as well. 
			
			When $d=2$, we proceed by contradiction. Assume that the wave operator $\mathcal{W}_{-}^l$ is bounded on $L^1(\mathbb{R}^2)$. Then $\mathcal{W}_{-}^lf(x)\in L^1(\mathbb{R}^2)$ for all $f\in L^1$. 
			Consider the  function
			$f = \left( \chi \left( 2|\cdot| \right) \right)^{\vee} \in L^1(\mathbb{R}^2)$, where $\chi$ is as defined in \eqref{eq-cut-off function}. From   \eqref{eq-T-phi-5-5}, we obtain that
			\begin{equation}\label{eq-rep-pair}
				\begin{aligned}
					\int_{\mathbb{R}^2} \mathcal{W}_{-}^lf(x) \d x= |\hat{\varphi}(0)|^2\lim_{\varepsilon\downarrow0}\int_{\mathbb{R}^2}\frac{G_+^{\alpha}(|\eta|)\chi(|\eta|) }{|\eta|^2-i\varepsilon}
					\,\d\eta, 
				\end{aligned}
			\end{equation}
			where $\hat{\varphi}(0)=\int{\varphi(x)\,\d x} \neq 0$.  However, by \eqref{eq: expansion for G when d=2}, the integral on the RHS of  \eqref{eq-rep-pair} diverges. This contradicts our assumption that $\mathcal{W}_{-}^l$ is $L^1$ bounded, thereby completing the proof by contradiction.
		\end{proof}

		\section{Proof of Theorem \ref{thm-main result-finite rank}}\label{sec5}
		In this section, we aim to prove Theorem \ref{thm-main result-finite rank}. 
		Recall that the stationary representation of wave operators for $W_-(H,H_0)$:
		\begin{equation}\label{eq-W1W2}
			\begin{split}
				W_-(H,H_0) f
				&=f-\sum_{i,j=1}^N\frac{1}{\pi i}\int_0^\infty \lambda g^+_{i,j}(\lambda) R_0^+(\lambda^2)\varphi_{j}\langle(R_0^+(\lambda^2)-R_0^-(\lambda^2))f,\varphi_i\rangle \d\lambda\\
				&:= f- \sum_{i,j=1}^N\mathcal{W}_{-,i,j}(H,H_0)f.
			\end{split}
		\end{equation}
		Note that $\mathcal{W}_{-,i,j}(H,H_0)f$ admits the same integral representation as the rank-one wave operator ${W}_{-}$, but with $G_+^\alpha(\lambda)$ replaced by $g_{i,j}^+(\lambda)$.
		Combining the  identity \eqref{eq-W1W2} with the proofs for rank-one perturbations in  Section \ref{sec3} and Section \ref{sec4}, the key remaining task is to verify that $g^+_{i,j}(\lambda)$ satisfies the same properties established for  $G_+^{\alpha}(\lambda)$  in Lemmas \ref{lm3.2}, \ref{lem:est for high G} and \ref{lem: for int donse not vanish}.

		We begin by considering the high-energy component.
		
		\begin{lemma}\label{lemma-g_ij-high energy}
			Let $d\ge 1$. Assume that each $\varphi_i\ (i=1,\cdots,N)$ satisfies (i) of \textbf{Condition $H_\alpha$} and that the spectral assumption \eqref{eq-spectral condition-finite rank} holds. Then for  $\lambda>\lambda_0$ with any fixed $\lambda_0,$ we have 
			\begin{equation}\label{eq-g_ij-high energy}
				\Big|\frac{\d^k}{\d\lambda^k}g_{i,j}^\pm(\lambda)\Big|\lesssim_{\lambda_0} \begin{cases}
					1,& \quad k=0,\\
					\lambda^{-1}, &\quad k=1,\cdots,\big[\frac{d+3}{2}\big]. 
				\end{cases}                                      
			\end{equation}
		\end{lemma}
		\begin{proof}
			It follows from the decay assumptions on $\varphi_j$ and the limiting absorption principle in \eqref{eq: lim absor prin} that for all $\lambda>\lambda_0$,
			\begin{equation}\label{eq4.14.1}
				\Big|\frac{\d^k}{\d\lambda^k}f^{\pm}_{i,j}(\lambda)\Big|\lesssim_{\lambda_0}\lambda^{-1}, \,\,\,\quad \,\,\,~k=0, 1,\ldots,[\frac{d+3}{2}].
			\end{equation}
			Recall that 
			\begin{equation}\label{eq4.14.2}
				\left(g_{i,j}^{\pm}(\lambda)\right)_{N\times N}=A^{\pm}(\lambda)^{-1}=\frac{1}{\det{(A^{\pm}(\lambda))}}\text {adj}(A^{\pm}(\lambda^2)),
			\end{equation}
			where $\text{adj}(A^{\pm}(\lambda))$ is the adjugate matrix of $A^{\pm}(\lambda)$ defined in \eqref{eq: def of matrix A}.
			By assumption \eqref{eq-spectral condition-finite rank}, one has $|\det{A^{\pm}(\lambda)}|\ge c_0>0$. Combining this with \eqref{eq4.14.1} and \eqref{eq4.14.2},  there exists some positive constant $C=C(N, \lambda_0, \varphi_1,\ldots, \varphi_N)$ such that
			\begin{equation}\label{eq4.14.2.0}
				|g_{i,j}^{\pm}(\lambda)|\leq C,
			\end{equation}
			which proves \eqref{eq-g_ij-high energy} for $k=0$. Further, we notice that
			\begin{equation}\label{eq4.14.3}
				\frac{\d^k}{\d\lambda^k}A^{\pm}(\lambda)^{-1}=A^{\pm}(\lambda)^{-1}\sum_{s=1}^{k}\sum_{l_{m,1}+\ldots+l_{m,s}=k}
				C(l_{m,1},\cdots,l_{m,s})\prod_{m=1}^s\left[\frac{\d^{l_{m,s}}}{\d\lambda^{l_{m,s}}}(A^{\pm}(\lambda)) \cdot A^{\pm}(\lambda)^{-1}\right].
			\end{equation}
			This, together with \eqref{eq4.14.1}, proves \eqref{eq-g_ij-high energy} for $k=1,\ldots,[\frac{d+3}{2}]$.
		\end{proof}
		
		Now we shift our focus to the  low-energy component. For $d\ge 3$, the results are straightforward to present. 
		However, the cases of $d = 1$ and $d = 2$ require more careful analysis due to the singularities of $R_0^\pm(\lambda^2)$ at $\lambda=0$.
		These lower-dimensional cases will be addressed separately in the following discussion.

		\begin{lemma}\label{lemma:g_ij-low energy-1}
			Let $d\ge 3$. Assume  each $\varphi_i\ (i=1,\cdots,N)$ satisfies (i) of \textbf{Condition $H_\alpha$} and that the spectral assumption \eqref{eq-spectral condition-finite rank} holds. Then there exists a small $\lambda_0>0$ such that for any $0<\lambda<\lambda_0$,  
			$$  g_{i,j}^{\pm}(\lambda)=g_{i,j}^{\pm,0}+ g_{i,j}^{\pm,1}(\lambda),$$
			where $g_{i,j}^{\pm,0}$ is a constant and $g_{i,j}^{\pm,1}(\lambda)$ satisfies
			\begin{equation}\label{eq:g_ij-low energy-1}
				\Big|\frac{\d^k}{\d\lambda^k}g_{i,j}^{\pm,1}(\lambda)\Big|\lesssim \lambda^{1-k},\ \ \    k=0,1,\cdots,\left[\frac{d+3}{2}\right].  
			\end{equation}
		\end{lemma}
		\begin{proof}
			We first establish the asymptotic expansion of $(I+PR_0^\pm(\lambda^2)P)^{-1}$ around zero. Similar to Lemma \ref{lemma-F low energy-d>3},  by the expansions in \eqref{eq：exp for  reso-0-1}, for $d\ge 3$, the operator $I+PR_0^\pm(\lambda^2)P$ admits the following  expansion
			$$I+PR_0^\pm(\lambda^2)P=I+PG_0P+\mathcal{R}_d^{\pm}(\lambda)\quad \mbox{in}\quad L^2,$$
			where $G_0=(-\Delta)^{-1}$ is the operator with the integral kernel as $c_d|x-y|^{d-2}$.
			Here, using our decay assumption on $\varphi{_j}$ and \eqref{eq: est r_d>3},  $\mathcal{R}_d^{\pm}(\lambda)$ satisfies 
			\begin{equation}\label{eq：est for tilde r_d}
				\left\|\frac{\d^{k}\mathcal{R}_d^{\pm}(\lambda)}{\d\lambda^k} \right\|_{L^2-L^2}\lesssim \lambda^{1-k},\quad k=0,1,\ldots,\left[\frac{d+3}{2}\right].  	
			\end{equation}
			Since $I + P G_0 P$ is strictly positive on $L^2$, it follows that $I + P G_0 P$ is invertible on $L^2$. By the Neumann series expansion, there exists $\lambda_0 > 0$ such that for $0 < \lambda < \lambda_0$, the operator $I + P R_0^\pm(\lambda^2) P$ is invertible on $L^2$ with
			\begin{align*}
				\left(I + P R_0^\pm(\lambda^2) P\right)^{-1} &= \sum_{k=0}^\infty \left(I + P G_0 P\right)^{-1} \left[-\mathcal{R}_d^{\pm}(\lambda) \left(I + P G_0 P\right)^{-1}\right]^k \\
				&= \left(I + P G_0 P\right)^{-1} + L^\pm(\lambda),
			\end{align*}
			where $L^\pm(\lambda)$ satisfies the same estimate as in \eqref{eq：est for tilde r_d}. Thus, we obtain the decomposition
			$$
			\left(g_{i,j}^{\pm}(\lambda)\right)_{N \times N} = \left(g_{i,j}^{\pm,0}\right)_{N \times N} + \left(g_{i,j}^{\pm,1}(\lambda)\right)_{N \times N},
			$$
			where $\left(g_{i,j}^{\pm,0}\right)_{N \times N}$ corresponds to the operator $\left(I + P G_0 P\right)^{-1}$ with 
			$$
			g_{i,j}^{\pm,0} = \langle \left(I + P G_0 P\right)^{-1} \varphi_i, \varphi_j \rangle,
			$$
			and $\left(g_{i,j}^{\pm,1}(\lambda)\right)_{N \times N}$ corresponds to $L^\pm(\lambda)$ with
			$$
			g_{i,j}^{\pm,1}(\lambda) = \langle L^\pm(\lambda) \varphi_i, \varphi_j \rangle.
			$$
			The results of this lemma follow directly from this decomposition.
		\end{proof}
		
		Next, we turn to the cases $d=1,2$.
		Recall that in the rank-one case, the $L^p$-boundedness  at the endpoints $p = 1$ and $p = \infty$ depends on whether $\int_{\mathbb{R}^d} \varphi(x) \, dx = 0$.
		Without loss of generality, we assume there exists $0 \leq k_0 \leq N$ such that
		\begin{equation}\label{eq4.16.1}
			\int_{\mathbb{R}^d} \varphi_i(x) \, \d x \neq 0 \quad \text{for } 1 \leq i \leq k_0; \qquad 
			\int_{\mathbb{R}^d} \varphi_i(x) \, \d x = 0 \quad \text{for } k_0 < i \leq N.
		\end{equation}
		Define  the quantity
		$$\sigma = \left( \sum_{j=1}^N \left(\int_{\R^d} \varphi_j(x) \, \d x\right)^2 \right)^{1/2}.$$
		\textbf{Key observations: } (i) $\sigma = 0$ if and only if $k_0 = 0$. (ii) If $k_0 > 0$, we point out that the analysis can be simplified by normalizing to the case $k_0=1$. More precisely, set
		$$
		\psi_1 = \frac{1}{\sigma} \sum_{j=1}^N \left( \int \varphi_j(x) \, \d x \right) \varphi_j\ne 0.
		$$
		Let $\{\psi_j\}_{j=2}^N$ be an orthonormal basis of the subspace $\{\psi_1\}^\perp \cap PL^2(\mathbb{R}^d)$. Then $\{\psi_j\}_{j=1}^N$ forms an orthonormal basis of $PL^2(\mathbb{R}^d)$. This allows us to rewrite $P$ as
		$$
		P = \sum_{j=1}^N \langle \cdot, \psi_j \rangle \psi_j.
		$$
		Since $\psi_j(x)=\sum_{k=1}^Na_{j,k}\varphi_k$, where $a_{j,k}\in \R$, $2\le j\le N$. It follows that conditions  \eqref{eq-F matrix}--\eqref{eq-spectral condition-finite rank} remain valid when  replacing $\varphi_j$ with $\psi_j$. Moreover, we have
		$$
		\int_{\mathbb{R}^d} \psi_j(x) \, \d x =\sum_{k=1}^Na_{j,k}\int \varphi_j(x) \, \d x=\sigma \langle \psi_j, \psi_1 \rangle= 0 \quad \text{for } j = 2,\ldots,N.
		$$
		This construction implies that all cases with $k_0 \geq 2$ are equivalent to the case $k_0 = 1$ and proves the observation above.
		
		Thus we  restrict our analysis to the cases $k_0 = 0$ (addressed in Lemma \ref{lemma:g_ij-low energy-2}) and $k_0 = 1$ (addressed in Lemma \ref{lemma:g_ij-low energy-3}) within \eqref{eq4.16.1}.

		\begin{lemma}\label{lemma:g_ij-low energy-2}
			Let $d=1$ or $d=2$. Assume that $k_0=0$ in  \eqref{eq4.16.1}. Then there exists some small $\lambda_0>0$ such that  for $0<\lambda<\lambda_0$,
			\begin{equation}\label{eq:g_ij-low energy-2}
				\Big|\frac{\d^k}{\d\lambda^k}g_{i,j}^{\pm}(\lambda)\Big|\lesssim \lambda^{-k},\ \ \    k=0,1,2.
			\end{equation}
		\end{lemma}
		\begin{proof}
			For  $d = 1$ with $\int_{\mathbb{R}^d} \varphi_i(x) \, \d x = 0$  for all $1\le i \le N$, we use the expansion \eqref{eq: exp for reso d=1} to express $f_{i,j}^{\pm}(\lambda^2)$ as
			$$
			f_{i,j}^{\pm}(\lambda^2) = \langle G_0 \varphi_i, \varphi_j \rangle + \iint r_{1}^{\pm}(\lambda|x - y|) \varphi_i(x) \varphi_j(y) \, \d x \d y,
			$$
			where $G_0$ denotes the operator with integral kernel $-\frac{1}{2}|x - y|$. By \eqref{eq: est for the int kernel of r_1}, we have, for $0 < \lambda < \frac{1}{2}$,
			\begin{equation}\label{eq: est for PrP}
				\left| \frac{\d^k}{\d\lambda^k} \left( f_{i,j}^{\pm}(\lambda^2) \right) \right| \lesssim \lambda^{-k}, \quad  \quad k = 0, 1, 2.
			\end{equation}
			Combining \eqref{eq4.14.2} with the spectral assumption \eqref{eq-spectral condition-finite rank} proves \eqref{eq:g_ij-low energy-2} for $d = 1$.
			
			The case $d = 2$ follows analogously by using \eqref{eq: exp for reso d=2} and \eqref{eq: exp for reso d=2-76}, whose details we omit. 
		\end{proof}

		\begin{lemma}\label{lemma:g_ij-low energy-3}
			Let $d=1$ or $d=2$. Assume that  $k_0=1$ in \eqref{eq4.16.1}.	Then there exists a small $\lambda_0>0$ such that for $0<\lambda<\lambda_0$, the following asymptotic expansions hold for all  $1\le i, j\le N$: 
			\begin{align}\label{eq: exp for g low d=1}
				g_{i,j}^{\pm}(\lambda)=
				\begin{cases}
					g_{i,j}^{\pm,0}(\lambda)+\lambda g_{i,j}^{\pm,1}+\lambda^2 g_{i,j}^{\pm,2}(\lambda),&\text{if}\,\, d=1,\\[4pt]
					g_{i,j}^{\pm,0}(\lambda)+(\log\lambda)^{-1}g_{i,j}^{\pm,1}+(\log\lambda)^{-2}g_{i,j}^{\pm,2}(\lambda),&\text{if}\,\,d=2,\\[4pt]
				\end{cases}
			\end{align} 
			where $g_{i,j}^{\pm,1}$ ($1\le i, j\le N$) are constants such that 
			\begin{equation}\label{eq:what is D_1^pm-0}
				0\neq g_{1,1}^{\pm,1}= 
				\begin{cases}
					\mp 2i(\int \varphi_1)^{-2} , &\quad n=1,\\ 
					-2\pi (\int \varphi_1)^{-2} , &\quad n=2,\\ 
				\end{cases}
			\end{equation}  
			and
			\begin{equation}\label{eq:g_ij0}
				g_{i,j}^{\pm,0}(\lambda)\equiv0, \quad \text{if\,\, either}\,\,\, i=1\,\,\, \text{or}\,\,\, j=1.
			\end{equation}
			Moreover, for $l=0,2$,
			\begin{equation}\label{eq:g_ij-low energy-3}
				\begin{split}
					\Big|\frac{\d^k}{\d\lambda^k}g_{i,j}^{\pm,l}(\lambda)\Big|\lesssim \lambda^{-k}, \quad k=0,1,2.
				\end{split}
			\end{equation}
		\end{lemma}
		\begin{proof}
			The key is to  establish the asymptotic expansion of $(I + P R_0^\pm(\lambda^2) P)^{-1}$ near zero energy. To overcome the singularity of $R_0^\pm(\lambda^2)$ at $\lambda=0$, our key idea is to first decompose the space $PL^2$ into an appropriate direct sum and then transform the invertibility problem into finding the inverse of the corresponding matrix operator.
			
			We begin with the case $d = 1$. Using the expansion \eqref{eq: exp for reso d=1}, we obtain
			\begin{equation}\label{eq4.19}
				I + P R^{\pm}_0(\lambda^2) P = \frac{\pm i}{2\lambda} P L_1 P + (I + P G_0 P) + \mathcal{R}_1^{\pm}(\lambda),
			\end{equation}
			where 
			\begin{equation*}
				L_1 f := \int_{\mathbb{R}} f(x) \, \d x, \quad 
				G_0 f := -\frac{1}{2} \int_{\mathbb{R}} |x - y| f(y) \, \d y,
			\end{equation*}
			and
			\begin{equation*}
				\mathcal{R}^{\pm}_1(\lambda) f := P \int_{\mathbb{R}} r^{\pm}_1(\lambda, |x - y|) P f(y) \, \d y.   
			\end{equation*}
			By \eqref{eq-decay condition} and  \eqref{eq: est for the int kernel of r_1}, we have 
			\begin{align*}
				\left\| \partial_\lambda^{k} \mathcal{R}^{\pm}_1(\lambda) \right\|_{L^2 \to L^2} \lesssim \lambda^{1 - k}, \quad 0 < \lambda < 1/2, \quad k = 0, 1, 2.
			\end{align*}
			Define two projections
			$$
			P_1 = \langle \cdot, \varphi_1 \rangle \varphi_1 \quad \text{and} \quad P_2 = \sum_{j=2}^N \langle \cdot, \varphi_j \rangle \varphi_j.
			$$
			Since   $k_0=1$ in \eqref{eq4.16.1}, we thus obtain that
			\begin{equation}\label{eq4.19-77}
				\int_\mathbb{R} P_2 f \, \d x = 0 \quad \text{and} \quad P L_1 P_2 = P_2 L_1 P = 0.
			\end{equation}
			For each $\lambda >0$, we define the isomorphism $T = (P_1, \lambda^{-1/2} P_2) : P_1 L^2 \oplus P_2 L^2 \to PL^2$ by
			$$
			Tf = f_1 + \lambda^{-1/2} f_2,\qquad f= (f_1, f_2).
			$$
			Let $T^*$ be the dual operator of $T$. If the matrix operator $T^*(I + P R_0^\pm(\lambda^2) P)T$  is invertible on $ P_1 L^2 \oplus P_2 L^2$, then $I + P R_0^\pm(\lambda^2) P$ is invertible on $PL^2$ and  (see \cite[Lemma 3.12]{JK})
			\begin{equation}\label{eq: inverse identity}
				(I + P R_0^\pm(\lambda^2) P)^{-1} = T \left(T^*(I + P R_0^\pm(\lambda^2) P)T\right)^{-1} T^*.
			\end{equation}
			Now we investigate the inverse of  $T^*(I + P R_0^\pm(\lambda^2) P)T$. Notice that by  \eqref{eq4.19} and \eqref{eq4.19-77}, the matrix  operator $T^*(I + P R_0^\pm(\lambda^2) P)T$ admits the expansion
			\begin{equation}\label{eq:operator-expansion}
				\begin{split}
					T^*(I + P R_0^\pm(\lambda^2) P)T &= \lambda^{-1} 
					\begin{pmatrix}
						\frac{\pm i}{2} P_1 L_1 P_1 & 0 \\
						0 & P_2 + P_2 G_0 P_2
					\end{pmatrix} 
					+ \lambda^{-1}
					\begin{pmatrix}
						\gamma_{0,0}^\pm(\lambda) & \gamma_{0,1}^\pm(\lambda) \\
						\gamma_{1,0}^\pm(\lambda) & \gamma_{1,1}^\pm(\lambda)
					\end{pmatrix} \\
					&:= \lambda^{-1} A_0^\pm + \lambda^{-1} \Gamma^\pm(\lambda),
				\end{split}
			\end{equation}
			where 
			\begin{align*}
				\gamma_{0,0}^\pm(\lambda) &= \lambda \left[ P_1 (I + P_2 G_0 P_2) P_1 + P_1 \mathcal{R}_1^\pm(\lambda) P_1 \right], \\
				\gamma_{0,1}^\pm(\lambda) &= \lambda^{1/2} \left[ P_1 (I + P G_0 P) P_2 + P_1 \mathcal{R}_1^\pm(\lambda) P_2 \right], \\
				\gamma_{1,0}^\pm(\lambda) &= \lambda^{1/2} \left[ P_2 (I + P G_0 P) P_1 + P_2 \mathcal{R}_1^\pm(\lambda) P_1 \right], \\
				\gamma_{1,1}^\pm(\lambda) &= P_2 \mathcal{R}_1^\pm(\lambda) P_2.
			\end{align*}
			We make the following observations concerning $ A_0^\pm$: (i) The operator $P_1 L_1 P_1$ is  invertible on $P_1 L^2$;  (ii)  The operator  $P_2 + P_2 G_0 P_2 = P_2 + P_2 (-\Delta)^{-1} P_2$ is invertible on $P_2L^2$. Thus, it follows that  $A_0^\pm$ is invertible  on $P_1 L^2 \oplus P_2 L^2$. Furthermore, by the Neumann series expansion, there exists $\lambda_0 > 0$ such that for $0 < \lambda< \lambda_0$, the operator $T^*(I + P R_0^\pm(\lambda^2) P)T$ is invertible on $P_1 L^2 \oplus P_2 L^2$ with expansion
			\begin{align}\label{eq-PP-7-8-1}
				\left(T^*(I + P R_0^\pm(\lambda^2) P)T\right)^{-1} 
				= \lambda \sum_{k=0}^\infty (A_0^\pm)^{-1} \left(-\Gamma^\pm(\lambda) (A_0^\pm)^{-1}\right)^k.
			\end{align}

			We claim  that when $0 < \lambda < \lambda_0$,  the following expansion holds on $PL^2$:
			\begin{equation}\label{eq-low energy expansions for operator-d=1}
				(I + P R_0^\pm(\lambda^2) P)^{-1} = P_2 \mathscr{R}_1^{\pm,0}(\lambda) P_2 + \lambda D_1^\pm + \mathscr{R}_1^{\pm,2}(\lambda),
			\end{equation}
			where $D_1^\pm \in \mathbb{B}(L^2)$. Moreover,  for $0 < \lambda < \lambda_0$, one has
			\begin{align}\label{eq-R-0-7-8-1}
				\|\partial_\lambda^k \mathscr{R}_1^{\pm,0}(\lambda)\|_{L^2 \to L^2} &\lesssim \lambda^{-k},\qquad k=0,1,2,
			\end{align}
			and 
			\begin{align}\label{eq-R-1-7-8-2}  
				\|\partial_\lambda^k \mathscr{R}_1^{\pm,2}(\lambda)\|_{L^2 \to L^2} &\lesssim \lambda^{2-k},\qquad k=0,1,2.
			\end{align}
			Note that the matrix representation of the operator $(I + P R_0^\pm(\lambda^2) P)^{-1}$ with respect to the basis $\{\varphi_j\}_{j=1}^n$ is precisely given by $(g_{i,j}^{\pm}(\lambda))$, i.e., 
			$$
			(g_{i,j}^{\pm}(\lambda))=\left\langle\varphi_i,\, (I + P R_0^\pm(\lambda^2) P)^{-1}\varphi_j\right \rangle.
			$$
			Consequently, \eqref{eq: exp for g low d=1} (for $d=1$), \eqref{eq:g_ij0} and \eqref{eq:g_ij-low energy-3} follow directly from the expansion  \eqref{eq-low energy expansions for operator-d=1} combined with the estimates \eqref{eq-R-0-7-8-1}-\eqref{eq-R-1-7-8-2}.

			We are left to prove \eqref{eq-low energy expansions for operator-d=1}-\eqref{eq-R-1-7-8-2}. By \eqref{eq: inverse identity} and \eqref{eq-PP-7-8-1}, we have
			\begin{align}\label{eq: expansion for inverse-1}
				(I + P R_0^\pm(\lambda^2) P)^{-1} &= \lambda  T(A_0^\pm)^{-1}T^*- \lambda T(A_0^\pm)^{-1} \Gamma^\pm(\lambda) (A_0^\pm)^{-1}T^* 
				\nonumber\\
				&+  \lambda \sum_{k=2}^\infty T(A_0^\pm)^{-1} \left(-\Gamma^\pm(\lambda) (A_0^\pm)^{-1}\right)^kT^*.
			\end{align}
			For the first term on the right-hand side of \eqref{eq: expansion for inverse-1}, it follows from the definition of $(A_0^\pm)^{-1}$ (\eqref{eq:operator-expansion}) that
			\begin{equation}\label{eq-7-8-01}
				\lambda T(A_0^\pm)^{-1}T^* = P_2 (P_2 + P_2 G_0 P_2)^{-1} P_2 \mp 2i\lambda P_1 (P_1 L_1 P_1)^{-1} P_1.  
			\end{equation}
			For the second term, we further decompose $\Gamma^{\pm}(\lambda)$ into
			\begin{align*}
				\Gamma^{\pm}(\lambda)=\Gamma_0^{\pm}(\lambda)+\Gamma_1^{\pm}(\lambda),
			\end{align*}
			where
			\begin{align*}
				\Gamma_0^{\pm}(\lambda)&=\begin{pmatrix}
					\lambda P_1 (I+P_2 G_0P_2 ) P_1  &   \lambda^{\frac12}P_1 (I+PG_0P) P_2  \\
					\lambda^{\frac12}P_2 (I+PG_0P) P_1             	& 0
				\end{pmatrix},
			\end{align*}
			and
			\begin{align*}
				\Gamma_1^{\pm}(\lambda)&=\begin{pmatrix}
					\lambda P_1 \mathcal{R}^{\pm}_1 (\lambda)P_1  &   \lambda^{\frac12}P_1 \mathcal{R}^{\pm}_1(\lambda) P_2  \\
					\lambda^{\frac12}P_2 \mathcal{R}^{\pm}_1(\lambda) P_1                    
					& P_2 \mathcal{R}^{\pm}_1(\lambda) P_2 
				\end{pmatrix}.  \\ 
			\end{align*}
			A direct computation yields that
			\begin{equation}\label{eq-7-9-01}
				\lambda T(A_0^\pm)^{-1} \Gamma_0^\pm(\lambda) (A_0^\pm)^{-1}T^*=\lambda P_1d_{1,1}^\pm P_2 + \lambda P_2d_{1,2}^\pm P_1+\lambda^2P_1d_{1,2}^\pm P_1,
			\end{equation}
			and 
			\begin{equation}\label{eq-7-9-02}
				\lambda T(A_0^\pm)^{-1} \Gamma_1^\pm(\lambda) (A_0^\pm)^{-1}T^*=P_2\mathscr{T}_1^{\pm,1}(\lambda)P_2+\mathscr{T}_2^{\pm,2}(\lambda),
			\end{equation}
			where $d_{1,s}^\pm$, $\mathscr{T}_1^{\pm,s}(\lambda)$ ($s=1,2,3$) are bounded operators  on $L^2$.  
			Moreover,  for $0 < \lambda < \lambda_0$, one has
			\begin{align}\label{eq-R-0-7-8-21}
				\|\partial_\lambda^k \mathscr{T}_1^{\pm,1}(\lambda)\|_{L^2 \to L^2} &\lesssim \lambda^{1-k},\qquad k=0,1,2,
			\end{align}
			and 
			\begin{align}\label{eq-R-1-7-8-22}  
				\|\partial_\lambda^k \mathscr{T}_1^{\pm,2}(\lambda)\|_{L^2 \to L^2} &\lesssim \lambda^{2-k},\qquad k=0,1,2.
			\end{align}
			We remark that by \eqref{eq-7-8-01} and \eqref{eq-7-9-01}, there is 
			\begin{equation}\label{eq:what is D_1^pm-1}
				D_{1}^{\pm}=\mp 2iP_1 (P_1 L_1 P_1)^{-1} P_1-P_1d_{1,1}^\pm P_2 + \lambda P_2d_{1,2}^\pm P_1.    
			\end{equation}
			This identity implies that 
			\begin{equation}\label{eq:what is D_1^pm-2}
				\big\langle D_{1}^{\pm}\varphi_1, \varphi_1\big\rangle= \mp 2i\big\langle (P_1 L_1 P_1)^{-1}\varphi_1,\, \varphi_1\big\rangle=\mp 2i(\int \varphi_1)^{-2}.    
			\end{equation}
			For the higher order terms ($k\ge 2$) on the right-hand side of  \eqref{eq: expansion for inverse-1}, one can use the following formal calculation: 
			\begin{equation*}
				\begin{pmatrix}
					\lambda & \lambda^{1/2} \\
					\lambda^{1/2} & \lambda
				\end{pmatrix}^k 
				= \begin{pmatrix}
					O(\lambda) & O(\lambda^{3/2}) \\
					O(\lambda^{3/2}) & O(\lambda)
				\end{pmatrix} \quad \text{for } k \geq 2 \text{ and } 0 < \lambda \ll 1,
			\end{equation*}
			and thus they have the same form as in \eqref{eq-7-9-02}, but have better bounds with respect to $\lambda$. 
			
			Therefore, \eqref{eq-low energy expansions for operator-d=1}-\eqref{eq-R-1-7-8-2} follow by combining \eqref{eq: expansion for inverse-1}--\eqref{eq-R-1-7-8-22}.
			
			
			
			Now we prove for  $d = 2$. By \eqref{eq: exp for reso d=2}, we have
			\begin{equation*}
				I + P R^{\pm}_0(\lambda^2) P =\left( \frac{\pm i}{4}-\frac{1}{2\pi}\gamma+\frac{1}{2\pi}\log2-\frac{1}{2\pi}\log\lambda\right)PL_1P+I+PG_0P+\mathcal{R}_2^\pm(\lambda),	
			\end{equation*}
			where $\mathcal{R}_2^\pm(\lambda)$ satisfies, for $0 < \lambda < \frac12$,
			\begin{align*}
				\left\| \partial_\lambda^{k} \mathcal{R}^{\pm}_2(\lambda) \right\|_{L^2 \to L^2} \lesssim \lambda^{1 - k}, \ \quad k = 0, 1, 2.
			\end{align*}
			We denote by $T=(P_1, (-\log\lambda)^{\frac 12}Q_2)$.
			As in dimension $d=1$, we have  
			\begin{align}\label{eq:operator-expansion-1}
				T^*(I + P R_0^\pm(\lambda^2) P)T &= (-\log\lambda)
				\begin{pmatrix}
					\frac{1}{2\pi}P_1 L_1 P_1 & 0 \\
					0 & P_2 + P_2 G_0 P_2
				\end{pmatrix} 
				+ (-\log\lambda)
				\begin{pmatrix}
					\gamma_{0,0}^\pm(\lambda) & \gamma_{0,1}^\pm(\lambda) \\
					\gamma_{1,0}^\pm(\lambda) & \gamma_{1,1}^\pm(\lambda)
				\end{pmatrix} \nonumber\\
				&:= (-\log\lambda) A_0^\pm +(-\log\lambda)\Gamma^\pm(\lambda),
			\end{align}
			where 
			\begin{align*}
				\gamma_{0,0}^\pm(\lambda) &=  (-\log\lambda)^{-1}\left[c^{\pm}P_1 L_1 P_1+ P_1 (I + P_2 G_0 P_2) P_1 + P_1 \mathcal{R}_2^\pm(\lambda) P_1 \right], \\
				\gamma_{0,1}^\pm(\lambda) &= (-\log\lambda)^{-\frac 12} \left[ P_1 (I + P G_0 P) P_2 + P_1 \mathcal{R}_2^\pm(\lambda) P_2 \right], \\
				\gamma_{1,0}^\pm(\lambda) &= (-\log\lambda)^{-\frac 12}\lambda^{1/2} \left[ P_2 (I + P G_0 P) P_1 + P_2 \mathcal{R}_1^\pm(\lambda) P_1 \right], \\
				\gamma_{1,1}^\pm(\lambda) &= (-\log\lambda)^{-1}P_2 \mathcal{R}_1^\pm(\lambda) P_2.
			\end{align*}
			Here, $c^\pm= \frac{\pm i}{4}-\frac{1}{2\pi}\gamma+\frac{1}{2\pi}\log2$.  Following the same analysis as in the one-dimensional case, it follows that $A_0^\pm$ is invertible on $P_1 L^2 \oplus P_2 L^2$. Furthermore, for $0 < \lambda < \lambda_0$ with a sufficiently small $\lambda_0 > 0$, the operator $(I + P R_0^\pm(\lambda^2) P)^{-1}$ admits the expansion
			\begin{equation}\label{eq-low energy expansions for operator-d=1-b}
				(I + P R_0^\pm(\lambda^2) P)^{-1} = P_2 \mathscr{R}_2^{\pm,0}(\lambda) P_2 + (-\log\lambda)^{-1} D_2^\pm + \mathscr{R}_2^{\pm,1}(\lambda),
			\end{equation}
			and
			\begin{align*}
				\|\partial_\lambda^k \mathscr{R}_2^{\pm,0}(\lambda)\|_{L^2 \to L^2} &\lesssim \lambda^{-k}, \ \ \ \ \ \ 
				\|\partial_\lambda^k \mathscr{R}_2^{\pm,1}(\lambda)\|_{L^2 \to L^2} \lesssim \lambda^{2-k},\quad 0\le k\le 2.
			\end{align*}
			This yields \eqref{eq: exp for g low d=1} for  $d=2$ as well as properties \eqref{eq:g_ij0}--\eqref{eq:g_ij-low energy-3}. Similar to 
			\begin{equation}\label{eq:what is D_2^pm}
				\big\langle D_{2}^{\pm}\varphi_1, \varphi_1\big\rangle= -2\pi\big\langle (P_1 L_1 P_1)^{-1}\varphi_1,\, \varphi_1\big\rangle= -2\pi (\int \varphi_1)^{-2}.    
			\end{equation}
			
			Therefore, the proof of the lemma is complete.
		\end{proof}
		
		Having established all the properties of $g^+_{i,j}(\lambda)$,  we are in the position to prove Theorem \ref{thm-main result-finite rank}.
		
		\begin{proof}[Proof of Theorem \ref{thm-main result-finite rank}]
			We decompose $\mathcal{W}_{-,i,j}(H,H_0)$ (see \eqref{eq-W1W2}) into high and low energy components:
			\begin{align*}
				\mathcal{W}_{-,i,j}^{h} &= \frac{1}{\pi i}\int_0^\infty \lambda(1-\chi(\lambda)) g^+_{i,j}(\lambda) R_0^+(\lambda^2)\varphi_j \langle (R_0^+(\lambda^2)-R_0^-(\lambda^2))f, \varphi_i \rangle \d\lambda, \\
				\mathcal{W}_{-,i,j}^{l} &= \frac{1}{\pi i}\int_0^\infty \lambda \chi(\lambda) g^+_{i,j}(\lambda) R_0^+(\lambda^2)\varphi_j \langle (R_0^+(\lambda^2)-R_0^-(\lambda^2))f, \varphi_i \rangle \d\lambda.
			\end{align*}
			
			Following exactly the same approach as in the rank one case, we  analyze each component separately and point out that Lemmas  \ref{lemma-g_ij-high energy}-\ref{lemma:g_ij-low energy-3} can serve as replacements for  Lemmas \ref{lm3.2}, \ref{lem:est for high G} and \ref{lem: for int donse not vanish}  used previously. More precisely:
			\begin{enumerate}
				\item \textbf{High energy component:}\\
				Lemma  \ref{lemma-g_ij-high energy} indicates that $ g^+_{i,j}(\lambda)$ satisfies the same estimates as $G_+^{\alpha}(\lambda)$   in Lemma \ref{lem:est for high G}, thus we conclude that  $\mathcal{W}_{-,i,j}^{h}$ is bounded on $L^p(\mathbb{R}^d)$ for all $1 \leq p \leq \infty$. 
				\item \textbf{Low energy component:}\\
				The boundedness properties depend on the dimension $d$ and the value of $k_0$ in  \eqref{eq4.16.1}. Moreover,  as established in the \textbf{key observation} following \eqref{eq4.16.1}, we may restrict our analysis to  two cases: $k_0=0$ and $k_0=1$.
				\begin{itemize}
					\item For $d \geq 3$,   Lemma \ref{lemma:g_ij-low energy-1} (replacing  Lemma \ref{lm3.2}) ensures that $\mathcal{W}_{-,i,j}^{l}$ is bounded on $L^p(\mathbb{R}^d)$ for all $1 \leq p \leq \infty$ and $1 \leq i,j \leq N$;
					
					\item  For  $d=1,2$ with $k_0=0$,   Lemma \ref{lemma:g_ij-low energy-2} (analogous to  part (i) of Lemma \ref{lem: for int donse not vanish}) guarantees the $L^p$-boundedness of $\mathcal{W}_{-,i,j}^{l}$ for all $1\le  p \le \infty$ and $1 \leq i,j \leq N$;
					
					\item For $d=1,2$ with  $k_0=1$, if $2 \leq i,j \leq N$,   Lemma \ref{lemma:g_ij-low energy-3} shows that $g_{i,j}^\pm(\lambda)$ satisfies condition (i) in Lemma \ref{lem: for int donse not vanish}, so $\mathcal{W}_{-,i,j}^{l}$ remains 
					bounded on $L^p(\mathbb{R}^d)$ for all $1\le  p \le \infty$;
					
					\item  For $d=1,2$ with $k_0=1$, if $i=1$ and $j\neq1$ or $i\neq1$ and $j=1$,   Lemma~\ref{lemma:g_ij-low energy-3} shows $g_{i,j}^\pm(\lambda)$ satisfies condition (ii) (for $d=1$) or (iii) (for $d=2$) in Lemma~\ref{lem: for int donse not vanish}. This implies the boundedness of $\mathcal{W}_{-,i,j}^{l}$ on $L^p(\mathbb{R}^d)$ for all $1\le p \le \infty$ when $d=1$ and $1\le p <\infty$ when $d=2$, as established in Lemma~\ref{lem:for i=1 and jneq 1}.
					
					\item  For $d=1,2$ with $k_0=1$, if $i=1$ and $j=1$ ,   Lemma \ref{lemma:g_ij-low energy-3} shows that  $g_{i,j}^\pm(\lambda)$ satisfies conditions (ii) (for $d=1$) or (iii)  (for $d=2$) in Lemma \ref{lem: for int donse not vanish}. Consequently, $\mathcal{W}_{-,i,j}^{l}$ is bounded on $L^p(\mathbb{R}^d)$ for $1<p<\infty$ and bounded from $L^1$ to $L^{1,\infty}$, but unbounded on $L^1(\mathbb{R}^d)$.
				\end{itemize}
			\end{enumerate}
			
			By combining these results, we  complete the proof of Theorem \ref{thm-main result-finite rank}.
		\end{proof}
		
		\begin{lemma}\label{lem:for i=1 and jneq 1} 
			Let $k_0=1$ in \eqref{eq4.16.1}. 
			When either $i=1$ and $j\neq1$ or $i\neq1$ and $j=1$, the operator $\mathcal{W}_{-,i,j}^{l}$ is bounded on $L^p(\mathbb{R}^d)$ for all $1 \le p \le \infty$ if $d=1$ and $1 \le p <\infty$ if $d=2$.
		\end{lemma}
		\begin{proof}
			For  $d=1$, the proof is similar to the arguments used in Proposition~\ref{theorem-low energy-d=1,2}. We only consider the case $i=1$ and $j\neq 1$. Note the condition $\int_{\R} \varphi_j dx =0$.
			The kernel of $\mathcal{W}_{-,1,j}^{l}$ can be expressed as the difference of $ L_+(x,y)- L_-(x,y)$, where
			\begin{equation*}
				L_\pm(x,y)=\frac{i}{4\pi }\int_0^\infty e^{i\lambda(|x|\pm|y|)}\lambda^{-1} 
				g_{1,j}^{+}(\lambda)\chi(\lambda)E_j^\pm(\lambda;x,y)\d\lambda, 
			\end{equation*}
			with
			\begin{equation}\label{E-pm-503-1-1}
				E_j^\pm(\lambda; x,y)=\int_{\mathbb{R}}e^{i\lambda (|x-x_1|-|x|)}\varphi_1(x_1)\d x_1\int_{\mathbb{R}}e^{\pm i\lambda(|x_2-y|-|y|)}\varphi_j(x_2)\d x_2.
			\end{equation}
			Using the condition $\int_{\mathbb{R}}\varphi_j(x)\d x=0$, we rewrite $E_j^\pm$ as 
			\begin{align*}
				E_j^\pm(\lambda; x,y)=&\int_{\mathbb{R}}e^{i\lambda (|x-x_1|-|x|)}\varphi_1(x_1)\d x_1\int_{\mathbb{R}}\big(e^{\pm i\lambda(|x_2-y|-|y|)}-1\big)\varphi_j(x_2)\d x_2\nonumber\\
				=& i\lambda\int_{\mathbb{R}}e^{i\lambda (|x-x_1|-|x|)}\varphi_1(x_1)\d x_1\int_{\mathbb{R}}\int_0^{|y-x_2|-|y|}e^{i\lambda\theta_2}\varphi_j(x_2)\d\theta_2\d x_2.\\
			\end{align*}
			Observe that $g_{1,j}^{+}(\lambda)$ admits the decomposition:
			\begin{equation*}
				g_{1,j}^{+}(\lambda) = \lambda g_{1,j}^{\pm,1} + \lambda^2 g_{1,j}^{\pm,2}(\lambda).
			\end{equation*}
			Consequently, the integrand  $\lambda^{-1}g_{1,j}^{+}(\lambda)E_j^\pm(\lambda; x,y)$ also satisfy the estimate \eqref{eq-E-pm-0501-1}. Repeating the argument leading to \eqref{L-pm-0501-2}, we obtain
			\begin{equation*}
				|L_\pm(x,y)| \lesssim \langle |x| \pm |y| \rangle^{-\frac{3}{2}},
			\end{equation*}
			which establishes that $\mathcal{W}_{-,i,j}^{l}$ is an admissible operator.
			
			For dimension $d=2$,  $\mathcal{W}_{-,i,j}^{l}$ is bounded on $L^p$ for $1<p<\infty$ by the proof of Lemma \ref{lem-T-phi-0504} . We are left to consider the endpoint $p=1$.
			
			We first consider the case $i\neq 1$ and $j=1$. Note that each $\varphi_i$ satisfies $\int_{\R^2} \varphi_i(x) \d x =0$. Thus we can write the kernel of $\mathcal{W}_{-,i,j}^{l}$ as
			\begin{equation*}
				\begin{split}
					\mathcal{W}_{-,i,1}^l(x,y)
					=\frac{i}{ 8\pi }\int_{\mathbb{R}^4}\Big(\int_{\langle x\rangle}^{|x-x_1|}\int_0^\infty\lambda^2 g_{i,1}^{+}(\lambda)\chi(\lambda)(H^+_0)'(\lambda\theta_1)J_0(\lambda r_2)\d\lambda \d\theta_1\Big)\cdot\varphi(x_1)\varphi(x_2)\d x_1\d x_2,
				\end{split}
			\end{equation*}
			where  $r_2=|x_2-y|$.
			By the properties of the free resolvents in Section \ref{sec2}, one has 
			\begin{align*}
				&(H^+_0)'(z)=e^{iz}\frac{\eta_0(z)}{z}+ e^{iz}\frac{\eta_1(z)}{z^{\frac 12}}, \\  
				&J_0(z)=-2i\, e^{iz}\frac{J^+(z)}{z^\frac{1}{2}}+2ie^{-iz}\frac{J^{-}(z)}{z^\frac{1}{2}},
			\end{align*}
			where $\eta_{s}(z)$ ($s=0,1$) and $J^\pm(z)$ satisfy 
			\begin{equation}\label{eq:est for support funct-1111}
				\Big|\frac{\d^k \eta_s(z)}{\d z^k}\Big|+\Big|\frac{\d^k J^{\pm}(z)}{\d z^k}\Big|\lesssim |z|^{-k},\quad \,\,\,k\in\mathbb{N}_0.
			\end{equation}
			We first consider the oscillatory integral 
			$$k(\theta_1, r_2)=\int_0^\infty\lambda^2 g_{i,1}^{+}(\lambda)\chi(\lambda)(H^+_0)'(\lambda\theta_1)J_0(\lambda r_2)\d\lambda,$$
			which can be written as a sum of 
			$$\theta_1^{-1+\frac{s}{2}}r_2^{-\frac{1}{2}}\int_0^\infty e^{i\lambda(\theta_1\pm r_2)}\lambda^{1-\frac{s}{2}} g_{i,1}^{+}(\lambda)\eta_s(\lambda \theta_1)J^{\pm}(\lambda r_2)\chi(\lambda) \d\lambda.$$
			Applying Lemma \ref{oscillatory estimates}, we derive that 
			\begin{equation}\label{eq:est for k(theta_1, r_2)}
				|k(\theta_1, r_2)|\lesssim \frac{1}{\theta_1 \sqrt{  r_2}\langle\theta_1-r_2\rangle^{\frac 32}}+\frac{1}{ \sqrt{\theta_1r_2}\langle\theta_1-r_2\rangle^{2-}}.   
			\end{equation}
			Using the same argument to obtain \eqref{eq:to tilde k}, we have 
			\begin{align*}
				\big|\mathcal{W}_{-,i,1}^l(x,y)\big|
				&\les  \int_{\R^2}\int_{\R^2}\int_{\langle  x\rangle  }^{|x-x_1|} |k(\theta_1,r_2)|  \left| \varphi(x_2) \right|  
				\,\d\theta_1 \d x_2  \d x_1 \\
				&\les \int_{\R^2} \int_{0 }^{\infty}    \tilde{k}(\theta_1,r_2) \langle\theta_1-|x|\rangle  ^{-N} \left| \varphi(x_2)\right|  \d\theta_1 \d x_2,
			\end{align*}
			where $N=\frac{\delta}{2}>2$ and 
			$$\tilde{k}(\theta_1,r_2)=\frac{1}{\langle\theta_1 \rangle \sqrt{  r_2}\langle\theta_1-r_2\rangle^{\frac 32}}+\frac{1}{ \sqrt{\theta_1r_2}\langle\theta_1-r_2\rangle^{2-}}.$$
			This implies that 
			\begin{align*}
				\sup_{y} \Big\|\mathcal{W}_{-,i,1}^l(x,y)\Big\|_{L^1_x}
				&\les \sup_{y}\int_{\R^2} \int_{0}^{\infty}    \tilde{k}(\theta_1,r_2)\langle\theta_1\rangle   \left| \varphi(x_2)\right|  \d\theta_1 \d x_2 \\
				&\les  \int_{\R^2} \left| \varphi(x_2)\right|  \d x_2 \\
				&\les 1,
			\end{align*}
			which in turn yields the $L^1$-boundedness of $\mathcal{W}_{-,i,1}^l$.  
			
			In the case $i=1$ and $j\neq1$, noting that $\varphi_j $ satisfies $\int_{\R^2} \varphi_j(x) \d x =0$,  we thus write  
			$\mathcal{W}_{-,i,j}^{l}$ as
			\begin{equation*}
				\begin{split}
					\mathcal{W}_{-,1,j}^l(x,y)
					=\frac{i}{ 8\pi }\int_{\mathbb{R}^4}\Big(\int_{|y|}^{|x_2-y|}\int_0^\infty\lambda^2 g_{1,j}^{+}(\lambda)\chi(\lambda)H^+_0(\lambda r_1)J'_0(\lambda\theta_2)\d\lambda \d\theta_2\Big)\cdot\varphi(x_1)\varphi(x_2)\d x_1\d x_2,
				\end{split}
			\end{equation*}
			where  $r_1=|x_1-x|$.  And we have 
			\begin{align*}
				&H^+_0(z)=-4i\, e^{iz}\frac{\omega(z)}{z^{\frac 12}}, \\  
				&J_0'(z)=e^{iz}\frac{J_{1}^+(z)}{z^\frac{1}{2}}+2ie^{-iz}\frac{J_{1}^{-}(z)}{z^\frac{1}{2}},
			\end{align*}
			where  $\omega(z)$ and $J_{1}^\pm(z)$  satisfy the same estimates as in \eqref{eq:est for support funct-1111}. Then we consider the oscillatory integral of the following form 
			$$k(r_1, \theta_2)=\int_0^\infty\lambda^2 g_{1,j}^{+}(\lambda)\chi(\lambda)H^+_0(\lambda r_1)(J_0)'(\lambda\theta_2)\d\lambda,$$
			which can be written as a sum of
			$$r_1^{-\frac{1}{2}}\theta_2^{-\frac{1}{2}}\int_0^\infty e^{i\lambda(r_1\pm \theta_2)}\lambda \cdot g_{1,j}^{+}(\lambda)\omega(\lambda r_1)J_1^{\pm}(\lambda \theta_2)\chi(\lambda) \d\lambda,$$
			Applying Lemma \ref{oscillatory estimates} yields that 
			$$|k(r_1, \theta_2)|\lesssim \frac{1}{ \sqrt{r_1\theta_2}\langle r_1-\theta_2\rangle^{2-}}.$$
			Thus, we have
			$$\big|\mathcal{W}_{-,1,j}^l(x,y)\big|
			\les \int_{\R^2} \int_{0}^{\infty}    |k(r_1,\theta_2)| \langle \theta_2-|y|\rangle  ^{-N} \left| \varphi(x_1)\right|  \d\theta_2 \d x_1, $$
			which implies that the kernel $\mathcal{W}_{-,1,j}^l(x,y)$ is admissible.
		\end{proof}
		
		\begin{remark}
			For $d=2$, the $L^\infty$-boundedness of $\mathcal{W}_{-,i,1}^l$ is lost when $i \neq 1$ due to the presence of the  singular  term $\frac{1}{\theta_1 \sqrt{r_2}\langle \theta_1 - r_2 \rangle^{\frac{3}{2}}}$ in \eqref{eq:est for k(theta_1, r_2)}. However, we conjecture that this boundedness property may still hold under a more refined analysis of $\mathcal{W}_{-,i,1}^l(x,y)$.
		\end{remark}

		\bigskip
		\noindent{\bf Acknowledgments.}  S. Huang was supported by the National Natural Science Foundation of China under grants 12171178 and 12171442. Z. Wu was supported by the China Postdoctoral Science Foundation under grant 2022M721254. A.S. is supported by NSF-DMS Grant 2205931. 

		\providecommand{\bysame}{\leavevmode\hbox to3em{\hrulefill}\thinspace}
		\providecommand{\MR}{\relax\ifhmode\unskip\space\fi MR }
		\providecommand{\MRhref}[2]{%
			\href{http://www.ams.org/mathscinet-getitem?mr=#1}{#2}
		}
		\providecommand{\href}[2]{#2}

		
	\end{document}